\documentclass[11pt,reqno]{amsart}
\RequirePackage[table]	{xcolor}			
\definecolor{green1}{RGB}{0,107,28}
\definecolor{green2}{RGB}{17,85,35}
\definecolor{green3}{RGB}{0,77,20}
\definecolor{green4}{RGB}{33,166,68}
\definecolor{green5}{RGB}{60,166,88}

\definecolor{blue1}{RGB}{3,56,91}
\definecolor{blue2}{RGB}{16,51,73}
\definecolor{blue3}{RGB}{1,40,66}
\definecolor{blue4}{RGB}{35,109,157}
\definecolor{blue5}{RGB}{59,118,157}
\RequirePackage[
	final,							
	colorlinks=true,
	urlcolor=blue3,						
	linkcolor=blue3,					
	citecolor=green4,					
	hypertexnames=false,					
	hyperfootnotes=false,					
	]{hyperref}						

\usepackage{a4wide}
\usepackage{latexsym}
\usepackage{amsopn,amscd}
\usepackage{amsfonts}
\usepackage{amssymb}
\usepackage{amsthm}
\usepackage{amsmath}
\usepackage[all]{xy}
\usepackage{scalefnt}

\usepackage[capitalise,nameinlink,noabbrev]{cleveref}
\usepackage{hyperref}

\DeclareFontFamily{OT1}{pzc}{}
\DeclareFontShape{OT1}{pzc}{m}{it}{<-> s * [1.200] pzcmi7t}{}
\DeclareMathAlphabet{\mathpzc}{OT1}{pzc}{m}{it}

\def\mathcat{\mathpzc}

\def\Alg{\mathcat{Alg}}

\def\kk{\mathfrak k}
\def\hh{\mathfrak h}

\def\mm{\mathfrak m}
\def\ll{\mathfrak l}
\def\zz{\mathfrak z}

\def\Ho{\mathrm H}

\def\gg{\mathfrak g}

\def\GG{G}

\def\TT{\mathrm T}
\def\UU{\mathrm U}

\def\RR{\mathbb R}
\def\ZZ{\mathbb Z}

\def\Hom{\mathrm{Hom}}

\def\empha{\em}

\def\bra{[\,\cdot\, , \,\cdot\,]}
\def\pbra{\{\,\cdot\, , \,\cdot\,\}}

\def\Ho{\mathrm{H}}

\def\CC{\mathbb C}
\def\AA{\mathcal A}

\def\Id{\mathrm{Id}}

\def\hinn{\langle\,\cdot\, , \,\cdot \,\rangle}

\def\Rep{\mathrm{Rep}}
\def\Ad{\mathrm{Ad}}
\def\an{\mathrm{an}}
\def\alg{\mathrm{alg}}

\def\ssi{\mathrm{ssimple}}

\def\OO{\mathcal O}

\def\hho{\hh^o}
\def\mmo{\mm^o}
\def\GG{G}
\def\Rr{\mu_\RR-\mathrm{ss}}
\def\KKK{K}

\newcommand{\group}{K}
\newcommand{\cgroup}{G}
\newcommand{\lieal}{\mathfrak k}
\newcommand{\clieal}{\mathfrak g}

\newtheorem{thm}{Theorem}[section]

\newtheorem{prop}[thm]{Proposition}
\newtheorem{defi}[thm]{Definition}

\newtheorem{rema}[thm]{Remark}

\long
\def\MSC#1\EndMSC{\def\arg{#1}\ifx\arg\empty\relax\else
      {\par\narrower\noindent
      2020 Mathematics Subject Classification. #1\par}\fi}

\long
\def\KEY#1\EndKEY{\def\arg{#1}\ifx\arg\empty\relax\else
    {\par\narrower\noindent
      Keywords and Phrases: #1\par}\fi}

\title[Finite construction of self-duality and related moduli spaces over a 
surface] 
{
Finite-dimensional construction of
self-duality and related moduli spaces over a closed Riemann surface
as stratified holomorphic symplectic spaces}

\author{Johannes Huebschmann  }
\address{
\newline
Universit\'e de Lille - Sciences et Technologies 
\\
D\'epartement de Math\'ematiques\\
\newline CNRS-UMR 8524,
Labex CEMPI (ANR-11-LABX-0007-01)
\\
\newline
59655 Villeneuve d'Ascq Cedex, France\\
\newline
Johannes.Huebschmann@univ-lille.fr
 }

\date{\today}
\numberwithin{equation}{section}

\begin{document}
\setcounter{page}{1}

\begin{abstract}
\noindent
In terms of appropriate extended moduli spaces,
we develop a finite-dimensional construction of the
self-duality and related moduli spaces over a closed Riemann surface
as stratified holomorphic symplectic spaces
by  singular finite-dimensional holomorphic symplectic reduction.
\end{abstract}

\maketitle

\medskip
\centerline
{Dedicated to the memory of Peter Slodowy.}
\medskip

\MSC 

\noindent
Primary: 53D30 

\noindent
Secondary: 14D21 14L24 14H60 32S60 53D17 53D20 53D50 58D27 81T13
\EndMSC

\KEY
Self-duality moduli space,
analytic representation variety,
holomorphic symplectic reduction,
stratified holomorphic symplectic space,
stratified hyperk\"ahler space, extended moduli space
 \EndKEY
{\tableofcontents}

\section{Introduction} 

In \cite{MR887284}, N. Hitchin constructed,
by infinite-dimensional methods,
the moduli spaces of the self duality equations
over a closed Riemann surface.
Hitchin showed that, away from the singularities,
such a moduli space acquires the structure of a hyperk\"ahler manifold.
Thus, in itself, such a moduli space is an interesting object, not only of 
Riemannian geometry but also of symplectic geometry.
In favorable cases, it
has no singularities,
and the resulting hyperk\"ahler manifold
also exhibits interesting fundamental algebro-geometric properties.

We here address  the issue of singularities seriously:
A hyperk\"ahler manifold
has an underlying
 holomorphic symplectic K\"ahler manifold;
the hyperk\"ahler constraint
is a somewhat special one on the K\"ahler structure.
We offer a description of the singularities
for a class of moduli spaces including those 
of the self duality equations
in the realm of what we call
{\em stratified holomorphic symplectic spaces\/},
a stratified holomorphic symplectic space
with a single stratum being
a holomorphic symplectic manifold.
We derive this description from a purely finite-dimensional
construction for these moduli spaces realized, according to the
{\em nonabelian Hodge\/} correspondence, as 
spaces 
of representations in a complex reductive Lie group 
of the fundamental group 
of the surface
and twisted versions thereof.
In particular, 
Simpson extended Hitchin's original result
by
showing that polystable Higgs bundles correspond to solutions 
of the self duality equations \cite {MR1307297, MR1320603} 
(Hitchin-Kobayashi correspondence for Higgs bundles)
and
Corlette  \cite{MR965220}
and Donaldson  \cite{MR887285}
established a correspondence between solutions
of the self-duality equations and representations of the fundamental group
(Hitchin-Kobayashi correspondence for complex connections).

The papers \cite{MR1460627}, \cite{MR1370113}, \cite {MR1277051}, 
\cite{MR1470732}, 
building on \cite{MR1112494}
and \cite{MR1362845},
settle a similar issue: a purely finite-dimensional construction
of the moduli spaces of semistable holomorphic vector bundles
on a Riemann surface, possibly punctured,
and of generalizations thereof
as stratified symplectic spaces in the sense of
\cite{MR1127479}, realized,
according to the {\em Hitchin-Kobayashi\/} correspondence
for principal bundles on a Riemann surface,
 as spaces of twisted representations
of the fundamental group in a compact Lie group.
The construction
proceeds by ordinary symplectic reduction applied to
a finite-dimensional {\em extended moduli space\/}
arising from a  product of $2\ell$ copies of the Lie group
(the group $\UU(n)$ for the case of holomorphic  rank $n$  vector bundles)
where $\ell$ is the genus of the surface or,
in the presence of punctures, from a variant thereof.
This structure depends on the  Lie group, 
a choice of an invariant inner product on its Lie algebra,
and the topology of a corresponding bundle, 
but is independent of any complex structure on $\Sigma$.

Here we proceed in the same way:
According to the nonabelian Hodge correspondence,
we construct analytic twisted representation 
varieties associated with  the fundamental group 
by holomorphic symplectic reduction
applied to
a finite-dimensional extended moduli space
arising from a  product of $2\ell$ copies of the corresponding
complexified  Lie group, a
complex reductive  Lie group.
Thus one can view such a twisted representation variety 
as a complexification of  a twisted representation 
space of the kind we explored in \cite{MR1370113};
since the polar decomposition and the inner product on the Lie algebra
determine a diffeomorphism between the total space of the real cotangent bundle
and the complexification of a compact Lie group, 
see Section \ref{twana} for details,
one can also view such a twisted representation variety
as the total space of a real cotangent bundle
(beware: in the presence of singularities the interpretation of the term 
cotangent bundle is not immediate)
of a twisted representation 
space of the kind we examined in \cite{MR1370113}.
Our main result, Theorem \ref{main}, says
that the complex structure of the Lie group,
a chosen invariant inner product on the Lie algebra,
and a certain additional ingredient which corresponds to the topology
of an associated bundle
determine a stratified holomorphic symplectic  structure
on a twisted representation variety
of the kind we study in this paper.
Our approach includes in particular
a new construction of
a {\em Betti\/} moduli space
(in the terminology of \cite {MR1307297, MR1320603})
as an analytic space
and puts a stratified holomorphic symplectic structure
on such a space.
To avoid confusion we note that 
the terminology in \cite {MR1307297, MR1320603} 
is \lq\lq character variety\rq\rq\ for our 
 \lq\lq representation variety\rq\rq.

To carry out the requisite holomorphic symplectic
reduction and to extract structural information
on the reduced level
we extend results due to Mayrand \cite{mayrand}.
Mayrand, in turn, 
builds on  
\cite {MR1127479},
\cite[Theorem 2.1]{MR1468352},
 an analytic 
version of the Kempf-Ness theorem
due to Heinzer-Loose  
\cite[Introduction \S 1.3  p.~289, \S 3.3 Theorem p.~295]{MR1274117},
and a 
 holomorphic slice theorem \cite[\S 2.7 Theorem p.~292]{MR1274117},
\cite[Theorem 1.12 p.~100]{MR1314032}.
In the etale world, this kind of slice theorem goes back
to \cite{MR0342523}.
Mayrand works exclusively with hyperk\"ahler manifolds, 
and in Section \ref{reduc} we show that
 his arguments apply to the more general setting in the present paper.
Accordingly, 
Theorems \ref{holopois2}, \ref{hyper1}, \ref{orbit},  \ref{hololoc},
\ref{orbit2}, \ref{poisson}, \ref{hyper4}, \ref{hyper3}
parallel or extend results in \cite{mayrand}. 
Mayrand's crucial technical result \cite[Theorem 1.3]{mayrand} is a 
{\em holomorphic symplectic slice theorem\/}
for hamiltonian hyperk\"ahler manifolds---it gives a local normal form
for a holomorphic momentum mapping formally exactly of the same kind
as the {\em Guillemin-Sternberg-Marle\/} 
local normal form of a real momentum mapping
 \cite{MR767835, MR753857, MR859857}---and
Theorem \ref{hololoc} extends this observation to a
{\em holomorphic symplectic slice theorem\/}
for hamiltonian holomorphic symplectic K\"ahler manifolds.
In the algebraic setting,
\cite[Theorem 3]{MR2230091} already
establishes such a symplectic slice theorem for
an algebraic hamiltonian action of a reductive group
on a non-singular affine symplectic variety.

Narasimhan-Seshadri constructed the moduli spaces
of semistable holomorphic vector bundles
by geometric invariant theory as normal projective varieties \cite{MR0184252},
see also \cite{newstboo},
and Atiyah-Bott obtained these spaces by infinite dimensional
methods and showed a smooth dense stratum of such a space
acquires a K\"ahler manifold structure \cite{MR702806}; 
thus when there is only one
such stratum, a Narasimhan-Seshadri moduli space 
becomes a K\"ahler manifold in a natural way.
However it may happen that such a moduli space
is non-singular as a projective variety
but still exhibits singularities as a stratified symplectic space
in the sense that it has more than one stratum.
This happens, e.g.,
for the moduli space of semistable rank 2 degree zero
holomorphic vector bundles with zero determinant 
on a Riemann surface of genus 2:
This moduli space is a 3-dimensional complex projective space,
and the non-stable semistable points constitute a Kummer surface
\cite{naramntw};
the stratified symplectic Poisson structure
is defined everywhere but the symplectic structures live
only on the strata \cite{MR1938554}.
It is very likely that similar phenomena occur
for the 
self duality moduli spaces as
stratified holomorphic symplectic spaces.
This is presumably in particular true for example
for the self duality moduli space
that corresponds to the moduli space of semistable rank 2 degree zero
holomorphic vector bundles with zero determinant 
on a Riemann surface of genus 2.
We expect this moduli space to be
a complex manifold but to have more than one stratum
as a stratified holomorphic symplectic space.

At the risk of being repetitive we note that, in what follows,
the {\em singularities of a stratified complex analytic space\/}
are the points in the complement of the top stratum;
these are not necessarily the singularities of
that complex analytic space \cite{MR788175}, that is,
a point in the complement of the top stratum
is not necessarily a singularity
relative to the complex analytic structure.

\section{Group cohomology construction of a hamiltonian holomorphic
symplectic K\"ahler structure}
\label{grcoho}

To take care over the terminology:
A complex {\em reductive\/} Lie group
is an affine complex algebraic group 
that is reductive
in the sense that every rational representation is completely
reducible; also the terminology {\em linearly reductive\/}
is in the literature. Equivalently,
a complex reductive Lie group arises as the complexification
of a compact Lie group,
and an affine complex
complex algebraic group is reductive if and only if the
unipotent radical of its connected component of the identity 
(in the classical topology) is trivial.
Thus a complex reductive Lie group is not necessarily connected.

Let $\cgroup$ be a complex reductive Lie group and $\cdot$
 a non-degenerate $\CC$-valued invariant 
symmetric bilinear
form on its Lie algebra $\clieal$.
The Maurer-Cartan calculus in
\cite{MR1460627}, \cite[Section 1]{MR1370113}, \cite {MR1277051}, 
\cite{MR1470732}, \cite{MR1362845}
is then available  over $\CC$ for the group $\cgroup$.
To recall its ingredients, let $\AA$ denote the de Rham forms
and, for a differential form $\alpha$ on $\cgroup$,
let $\alpha_j$ ($j=1,2$) denote the pullback of $\alpha$
by the projection $p_j$ to the $j$'th component:
\begin{enumerate}
\item the left invariant Maurer-Cartan form
$\omega \in \AA(\cgroup,\clieal)$
and the right invariant Maurer-Cartan form
$\overline \omega \in \AA(\cgroup,\clieal)$; 
\item 
the triple product
$\tau(x,y, z) = \tfrac 12 [x,y]\cdot z$,
$x,y,z \in \clieal$;
\item 
the Cartan 3-form
$\lambda = \tfrac 1 {12} [\omega,\omega]\cdot \omega$;
\item the 2-form $\Omega= \tfrac 12 \omega_1 \cdot \overline \omega_2 \in \AA^2(\cgroup \times \cgroup)$;
\item the equivariant 1-form $\vartheta \colon \clieal \to \AA^1(\cgroup)$ given by
\begin{equation}
\vartheta(X) = \tfrac 12 X \cdot (\omega + \overline \omega),\ X \in \clieal.
\end{equation}

\end{enumerate}

Let
\begin{equation}
\mathcal P = \langle x_1,y_1,\dots,x_\ell,y_\ell; r\rangle,\quad
r = \Pi [x_j,y_j],
\end{equation}
be the
standard presentation of the
fundamental group $\pi$ of a closed (real) surface $\Sigma$
of genus $\ell$.
The relator $r$ induces a complex algebraic  map
\begin{equation}
r \colon \cgroup^{2 \ell} \longrightarrow \cgroup.
\end{equation}
Let $O\subseteq \clieal$ be the open $\cgroup$-invariant subset of $\clieal$
where the exponential mapping
from $\clieal$ to $\cgroup$
is regular;
the reader will notice that $O$ contains the center of $\clieal$.
Define the space
$\mathcal H(\mathcal P,\cgroup)$
by requiring that
\begin{equation}
\begin{CD}
\mathcal H(\mathcal P,\cgroup)
@>r_O>> O
\\
@V{\eta}VV
@VV{\mathrm{exp}}V
\\
\cgroup^{2\ell}
@>>r> \cgroup
\end{CD}
\label{PB}
\end{equation}
be a pullback diagram; here we denote by $\eta$ and $r_O$
the induced maps.
The space
$\mathcal H(\mathcal P,\cgroup)$
is a complex manifold
and the induced map
$\eta$ from
$\mathcal H(\mathcal P,\cgroup)$
to
$\cgroup^{2\ell}$
is a holomorphic codimension zero immersion whence
$\mathcal H(\mathcal P,\cgroup)$ has the same dimension as $\cgroup^{2\ell}$.

Let $F$ be the free group on  $x_1,y_1,\dots,x_\ell,y_\ell$.
Evaluation yields a bijection $\Hom(F,\cgroup) \to \cgroup^{2\ell}$.
This induces an injection
of $\mathrm{Hom}(\pi,\cgroup)$
into
$\mathcal H(\mathcal P,\cgroup)$
and, in this way, we view
$\mathrm{Hom}(\pi,\cgroup)$
as a subspace of
$\mathcal H(\mathcal P,\cgroup)$.

Let $c \in C_2(F)$ 
be a 2-chain
whose image in  $C_2(\pi)$ 
is closed and
represents  a generator of
$\mathrm H_2(\pi) \cong \mathbb Z$.
Our approach is independent of a choice of complex structure
on $\Sigma$ and hence we need not worry about the choice of
an orientation.
Let 
\begin{equation}
\omega_{c,\mathcal P} = \eta^*(\omega_c)-r^*B
\label{closed}
\end{equation}
 be the closed
$\cgroup$-invariant $2$-form
on $\mathcal H(\mathcal P,\cgroup)$ 
in  \cite[Theorem 1]{MR1370113},
let $\psi \colon \clieal \to   \clieal^*$
denote the adjoint of the $2$-form $\,\cdot\,$
on $\clieal$,
and recall that the composite
\begin{equation}
\mu_{c,\mathcal P} \colon
\mathcal H(\mathcal P,\cgroup)
\stackrel{r_O} \longrightarrow O \subseteq \clieal
\stackrel{\psi}\longrightarrow  \clieal^* 
\label{eq}
\end{equation}
is an  equivariantly closed extension 
of  $\omega_{c,\mathcal P}$
\cite[Theorem 2]{MR1370113}
(written there as $\mu$).
As for how $\psi$ arises in this context, see also the remark at the end of
Section 1 of \cite{MR1370113}.
By construction, under the present circumstances,
$\omega_{c,\mathcal P}$
and $\mu_{c,\mathcal P}$ are holomorphic.

Let
$\mathcal M(\mathcal P,\cgroup)$
be the subspace of
$\mathcal H(\mathcal P,\cgroup)$
where
the 2-form $\omega_{c,\mathcal P}$
is non-degenerate;
this is an open $\cgroup$-invariant
subset containing
the pre-image $r^{-1}(\zz)$.
Abusing the notation slighly,
denote the restriction of $\mu_{(c,\mathcal P)}$
to 
$\mathcal M(\mathcal P,\cgroup)$
as well by $\mu_{(c,\mathcal P)} \colon \mathcal M(\mathcal P,\cgroup) \to \clieal^*$.
Then
\begin{equation}
(\mathcal M(\mathcal P,\cgroup),\omega_{c,\mathcal P},\mu_{(c,\mathcal P)})
\end{equation}
is a $\cgroup$-hamiltonian complex manifold.

Applying the procedure of symplectic reduction naively
to $(\mathcal M(\mathcal P,\cgroup),\omega_{c,\mathcal P},\mu_{(\mathcal P,\cgroup)})$
poses problems since we need a \lq\lq good\rq\rq\ 
analytic (or Hilbert) $\cgroup$-quotient 
\cite{MR1748608}
of 
analytic sets of the kind
$\mu_{(\mathcal P,\cgroup)}^{-1}(q)$ for points $q$ in the dual
$\zz^*$ of the center $\zz$ of $\clieal$,
this dual $\zz^*$ being well-defined
since $\zz$ is a direct summand of $\clieal$.
In the next section we show how results in 
\cite{MR1274117} and
\cite{mayrand}
enable us to overcome these difficulties.

\begin{rema}
{\rm
An extended moduli space 
arises
as a special case of a general construction
which renders lattice gauge theory rigorous
\cite{MR1670408}.
}
\end{rema}

\section{Reduction of hamiltonian holomorphic symplectic K\"ahler manifolds}
\label{reduc}

For a smooth symplectic manifold with a hamiltonian action
of a compact Lie group, 
 Sjamaar-Lerman proved that
the reduced space acquires a stratified symplectic structure
 \cite{MR1127479}.
Their arguments 
rely on the {\em Guillemin-Sternberg-Marle\/} 
local normal form of the momentum mapping
 \cite{MR767835, MR753857, MR859857}.
Sjamaar-Lerman  \cite{MR1127479} 
noted that
this normal form implies that, locally,
such a reduced space
is isomorphic to one arising from linear symplectic reduction
and thereby extended
the {\em Darboux\/} theorem  to such reduced spaces.
Also, from the local model, they deduced that
the orbit type decomposition
is a {\em Whitney stratification\/}.

In \cite{mayrand}, Mayrand
addresses these issues in the holomorphic setting.
He settles them merely for hamiltonian 
hyperk\"ahler manifolds
but his arguments, suitably extended, work 
for hamiltonian
 holomorphic symplectic K\"ahler manifolds, and this extension clarifies the 
nature of the arguments, simplifies the exposition
and, as we show in this paper, opens a wealth of 
attractive
examples.
Here we extend this approach to hamiltonian
 holomorphic symplectic K\"ahler manifolds, taylored to our purposes.

\subsection{Decomposed and stratified spaces}

A {\em decomposed\/} space is a space $X$
together with a family of pairwise disjoint
subspaces that are smooth manifolds, the {\em pieces\/}
of the decomposition, such that $X$ is the union of the pieces.
For a  decomposed space $X$,
we use the notation $C^\infty(X)$ 
for an algebra of real-valued continuous functions on $X$,
a {\em smooth structure\/} 
\cite {MR0467544},
which, on each
piece of the decomposition, are ordinary smooth functions;
we then denote by $C^\infty(X,\CC)$
the obvious extension of
$C^\infty(X)$ to an algebra of
complex-valued continuous functions on $X$.
There is no claim to the effect
that the restriction $C^\infty(X) \to C^\infty(S)$
to a stratum $S$ be onto; in the situations under discussion below,
the image of the restriction will contain the compactly supported functions
on that stratum, and this suffices
for characterizing the various geometric structures under discussion;
thus, there is no need to \lq\lq sheafify\rq\rq\ the 
emooth structures.
Below we use the term \lq stratification\rq\ 
and \lq stratified\rq\ space
but, deliberately, 
we do not make this precise.
In particular,
 \lq stratified\rq\ space could simply mean
 \lq decomposed\rq\ space,
and the definitions still make sense.
Mather's definition \cite{MR0368064}
provides a good understanding of the idea of a statification; see also
\cite{MR1869601} and the literature there.
For intelligibility we recall that
a stratification (in the sense of Mather)
of a space $X$ is a map $\mathcal S$ which assigns to each point $x$ of $X$
the set germ $\mathcal S_x$ of a locally closed subset of $X$ such 
that the following holds:
{\em For each $x \in X$ there is an open 
neighborhood\/} $U$ {\em of\/} $x$ {\em and a decomposition $\mathcal Z_U$ of $U$ 
such that, for $y \in U$, the set germ $\mathcal S_y$ coincides 
with the set germ of the unique piece $R_y \in \mathcal Z_U$ 
which contains $y$ as an element.}

Recall that
a {\em stratified K\"ahler
space\/} \cite{MR2096203, MR2212880, MR2883413}
consists of a complex analytic space $X$, together with
\begin{enumerate}
\item[{\rm (i)}] a complex analytic stratification (a not necessarily proper refinement
of the standard
complex analytic stratification,
cf.
{\rm \cite{MR788175}}), and with
\item[{\rm (ii)}] a real stratified symplectic structure
$(C^{\infty}X,\{\,\cdot\,,\,\cdot\,\})$ \cite{MR1127479} which is
 compatible with the complex
analytic structure.
\end{enumerate}
The two structures being {\empha compatible\/}
means the following:
\\
(i) For each point $q$ of $X$ and each holomorphic function
$f$ defined on an open neighborhood $U$ of $q$,
there is an open neighborhood $V$ of $q$ with $V \subset U$
such that, on $V$,
$f$ is the restriction of a function in $C^{\infty}(X,\CC)$;
\\
(ii) on each stratum,
the symplectic structure
determined by the symplectic Poisson structure
(on that stratum) combines
with the complex analytic structure to a K\"ahler
structure.

We extend this terminology to the hyperk\"ahler setting as follows;
to this end we recall that the three K\"ahler forms of a hyperk\"ahler 
structure
encapsulate the entire hyperk\"ahler structure, 
cf. \cite [Lemma 6.8]{MR887284},
\cite {MR935967}.

\begin{defi}
\begin{enumerate}
\item
A {\em stratified Poisson hyperk\"ahler space\/}
consists of a stratified space $X$, a smooth structure $C^\infty (X)$ on $X$,
and three Poisson structures $\pbra_1$, 
$\pbra_2$,  $\pbra_3$
on  $C^\infty (X)$
so that, on each stratum, 
for $j =1,2,3$, the bracket $\pbra_j$ 
is the Poisson structure associated with a
symplectic structure $\omega_j$ and that
$\omega_1$, $\omega_2$, $\omega_3$
constitute a hyperk\"ahler structure on that stratum.

\item
A {\em stratified holomorphic symplectic space\/}
consists of a complex analytic space $(X,\OO_X)$ 
together with
a complex analytic stratification and
a holomorphic Poisson structure $\pbra_X$
on the sheaf $\OO_X$ of germs of holomorphic functions
on $X$ which, on each stratum,
restricts to the holomorphic Poisson structure
associated with a holomorphic symplectic structure
on that stratum.
\item
 A {\em stratified holomorphic symplectic K\"ahler
space\/}
consists of a stratified K\"ahler space $(X,C^\infty(X),\OO_X,\pbra_\RR)$, 
together with
a holomorphic Poisson structure $\pbra_\CC$
on the sheaf $\OO_X$ of germs of holomorphic functions
on $X$ which, on each stratum,
restricts to the holomorphic Poisson structure
associated with a holomorphic symplectic structure
on that stratum.
\item
 A {\em weak stratified hyperk\"ahler
space\/} is a
stratified holomorphic symplectic K\"ahler
space 
$(X,C^\infty(X),\OO_X,\pbra_\RR,\pbra_\CC )$ such that,
on each stratum, the pieces of structure combine to an
ordinary 
hyperk\"ahler structure.
\item
A  {\em stratified 
hyperk\"ahler space} 
is a stratified space 
together with {\rm (i)}
three complex analytic structures $\OO_{\mathsf I}$,
$\OO_{\mathsf J}$, $\OO_{\mathsf K}$
which are compatible with the stratification
and {\rm (ii)}  three pairwise compatible real Poisson structures
$\pbra_{\mathsf I}$,
$\pbra_{\mathsf J}$, $\pbra_{\mathsf K}$
such that
\begin{equation}
\left(\OO_{\mathsf I},\pbra_{\mathsf J}+i\pbra_{\mathsf K}\right), \quad
\left(\OO_{\mathsf J},\pbra_{\mathsf K}+i\pbra_{\mathsf I}\right), \quad
\left(\OO_{\mathsf K},\pbra_{\mathsf I}+i\pbra_{\mathsf J}\right)
\label{holopoisdef}
\end{equation}
are holomorphic Poisson structures
which are compatible with the stratification and, on each stratum, restrict
to an ordinary  hyperk\"ahler structure.
\end{enumerate}
\end{defi}
A stratified hyperk\"ahler structure generates
a sphere of complex analytic and compatible
real Poisson structures.

\begin{rema}
{\rm
Mayrand gives the definition of a stratified 
hyperk\"ahler space as \cite[Definition 3.1.9]{maxence2019a}.
}
\end{rema}

\subsection{Quotients}
\label{quotients}

Let $\GG$ be a 
topological group.
For a $\GG$-space $X$,
a $\GG$-{\em subset\/} is a subset of $X$ 
that  is closed under 
the $\GG$-action.
For  $\GG$-space $Y$,
we say a $\GG$-invariant map $\pi \colon Y \to Y_0$
to a space $Y_0$ (with trivial $\GG$-action) 
is a $\GG$-{\em reduction\/} if
\begin{enumerate}
\item[{(\rm $G$-red 1)}]
every fiber $\pi^{-1}(y_0) \subseteq Y$, as $y_0$ ranges over $Y_0$, contains
exactly one closed $\GG$-orbit.
\end{enumerate}

Let $Y$ be a $\GG$-space.
As in \cite[\S 1.1 p.~173]{MR364272},
consider the following property,
see also $(*)$ \cite[\S 7.2 p.~420]{MR1087217}:
\begin{enumerate}
\item[{\rm (A)}]
Each $\GG$-orbit in $Y$ contains in its closure
a unique closed $\GG$-orbit.
\end{enumerate}

Let $Y$ be a $\GG$-space enjoying property (A).  Extending a construction in
\cite[\S 1.1 p.~173]{MR364272},
see also  \cite[\S 7.2 p.~421]{MR1087217},
define the 
{\em quotient $Y // \GG$ of
$Y$ by\/} $\GG$ to be the space whose points are the
closed $\GG$-orbits in $Y$, the $\GG$-{\em quotient map 
$\pi_Y\colon Y \to Y // \GG$\/}
to be the map which assigns to a point of $Y$ 
the unique closed orbit in the closure of its $\GG$-orbit,
and endow $ Y // \GG$ with the quotient topology.
Then $\pi_Y\colon Y \to Y // \GG$
is a $\GG$-reduction.

\subsection{Holomorphic symplectic reduction}
\label{hsr}

Let $(M,\omega_\CC)$ be a holomorphic symplectic manifold,
let $G$ be a complex reductive Lie group, write its Lie algebra as
$\gg$, suppose $G$ acts holomorphically on $M$ in a Hamiltonian fashion,
and let $\mu_\CC \colon M \to \gg^*$ denote
the holomorphic momentum mapping.
We refer to $(M,\omega_\CC,\mu_\CC)$
as a $G$-{\em hamiltonian holomorphic symplectic manifold\/}.

Our goal is to build the analogue  of the 
stratified symplectic
structure
on the reduced space for the real case
recalled at the beginning of this section.
The present aim is to show that  the analytic variant of
Kempf-Ness theory in \cite{MR1408556} 
yields the 
requisite complex
analytic quotient of the 
zero locus $\mu_\CC^{-1}(0)$
as a complex analytic space.
To this end
suppose that $M$ possesses,
independently of  $\omega_\CC$,
 an ordinary real K\"ahler form $\omega_\RR$
invariant under a maximal compact subgroup $K$ of $G$, and 
suppose the $K$-action on $M$ is hamiltonian
with momentum mapping $\mu_\RR\colon M \to \kk^*$.
Consider the
subspace
\begin{align}
M^{\Rr}&=\{q \in M; \overline {Gq}\cap \mu_\RR^{-1}(0) \ne 
\emptyset\}
\label{semistable}
\end{align}
 of {\em momentum semistable points of
$M$\/} with respect to $\mu_\RR$, cf.
\cite[Section 0]{MR1408556} for the terminology;
these are the {\em analytically semistable points\/}
in the sense of \cite[Definition 2.2 p.~109]{MR1314032}.
The following summarizes various results in the literature.

\begin{prop} Suppose the subspace $M^{\Rr}$
of momentum semistable points in $M$ is non-empty.
\label{openness}
\begin{enumerate}
\item 
The subspace $M^{\Rr}$ is $\GG$-invariant 
and open in $M$,
indeed, the smallest 
 $\GG$-invariant 
open subspace of $M$ containing $\mu_\RR^{-1}(0)$.

\item
The zero locus
$\mu_\RR^{-1}(0)$
is
a {\em Kempf-Ness set (fiber critical set)\/}, 
that is,
\begin{enumerate}
\item[{(\rm KN 1)}]
for $x \in  M^{\Rr}$,
the orbit $\GG x$ is closed in  $M^{\Rr}$ if and only if
$\GG x \cap \mu_\RR^{-1}(0) \ne \emptyset$;
\item[{(\rm KN 2)}]
for $x \in \mu_\RR^{-1}(0)$, the $\KKK$-orbit $\KKK x$ coincides with
$\GG x \cap \mu_\RR^{-1}(0)$.
\end{enumerate}

\item The $\GG$-manifold
$M^{\Rr}$
 admits a $\GG$-reduction
$\pi\colon M^{\Rr} \to M^{\Rr}// \GG$
in such a way that the inclusion 
$\mu_\RR^{-1}(0) \subseteq M^{\Rr}$
induces a homeomorphism
\begin{equation}
M//_{\mu_\RR} K = \mu_\RR^{-1}(0)/K \to M^{\Rr}//G.
\end{equation}
In particular, the quotient space
$\mu_\RR^{-1}(0)/K \cong M^{\Rr}//G$ is a Hausdorff space.
\item
The subspace $M^{\Rr}$
is dense in $M$.

\end{enumerate}
\end{prop}

Claims (1) -- (3) are due to \cite{MR1274117}
(\S 1.3 Theorem p.~289, \S 3.3 Theorem p.~295).
Under the additional assumption that the gradient flow
of the negative of the norm square of $\mu_\RR$ be globally defined
they are in \cite[Proposition 2.4 p.~110, Theorem 2.5 p.~112]{MR1314032};
this assumption holds, e.g., when $\mu_\RR$ is proper.
Under even more restrictive hypotheses
these observations are due to \cite{MR766741}.
Claim (4) is \cite[Lemma in Section 9 p.~83]{MR1408556}.

The reasoning in  \cite{MR1274117} for 
the openness of $M^{\Rr}$ is somewhat cryptic.
This openness is certainly well understood among the experts.
However, the non-expert will have difficulties
extracting a proof from the literature.
We therefore take the liberty of sketching a proof,
concocted with the help of P. Heinzer.
A proof  substantially different from that we are about to reproduce 
is in \cite{heinznerstoetzel}.

Consider a Kaehler manifold $(X,\omega)$
with a holomorphic $G$-action
whose restriction to a maximal compact subgroup $K$
preserves $\omega$. Recall
a $K$-invariant function
$\varphi\colon X \to \RR$
is a {\em (Kaehler) potential\/}
when
\begin{align}
\omega&=- \tfrac 12 d d^c \phi = i \partial \overline \partial \phi,
\ 
d^c = i (\partial- \overline \partial);
\label{pot1}
\end{align}
then $\varphi$ is necessarily strictly plurisubharmonic and,
with the notation $\xi_X$ for the vector field 
on $X$ which a member $\xi$ of the Lie algebra $\kk$ of $K$ induces,
the identity
\begin{align}
\xi \circ \mu&= \tfrac 12  (d^c \phi)(\xi_X)= \tfrac 12  (d\phi)(J\xi_X), \ \xi \in \kk,
\label{pot2}
\end{align}
characterizes a $K$-momentum mapping $\mu \colon X \to \kk^*$
which renders the $K$-action on $X$ hamiltonian with respect to $\omega$.
For a Hausdorff $G$-quotient $\pi \colon X \to Q$ 
(provided it exists)
a {\em relative exhaustion\/}
 \cite[\S 3.1 p.~330, \S 3.3 p.~336]{MR1748608}
is a smooth $K$-invariant function
$\psi\colon X \to \RR$
that
is bounded from below and has the property that
\begin{equation}
\psi \times \pi \colon X \to \RR \times  Q
\end{equation}
is proper.
When $X$ is Stein, a Hausdorff quotient exists
as a Stein space
 \cite{MR656653}, 
\cite[Proposition 3.1.2 p.~328]{MR1748608}.
Let $N$ be a Stein manifold with a holomorphic
$G$-action and a strictly plurisubharmonic relative exhaustion function
$\varphi\colon N \to \RR$ 
invariant under a maximal compact subgroup $K$ of $G$.
We then say $(N,G,K,\varphi)$
is a {\em relative exhaustion
Stein $G$-manifold\/}.
Recall a real function $f$ 
on a (reasonable topological) space  $D$ 
is an {\em exhaustion function\/}
if $\{z; f(z) < r\} \subseteq D$
is relatively compact in $D$ for any real $r$.
Here is \cite[Lemma 1 p.~131]{MR1293876}
in another guise:

\begin{prop}
\label{relexhaust}
For a  relative exhaustion
Stein $G$-manifold
$(N,G,K,\psi)$,
the momentum semistable subspace
relative to the  momentum mapping
$\mu \colon N \to \kk^*$ 
which $\psi$ induces via 
{\rm \eqref{pot2}}
coincides with $N$.
\end{prop}

\begin{proof}
Since the restriction of $\psi$ to a fiber is proper and bounded,
it is an exhaustion on that fiber and, in view of \eqref{pot2},
the restriction of $\psi$
to a closed orbit  has a critical point, necessarily an 
absolute minimum.
Hence,
with respect to the Stein quotient map $\pi \colon N \to Q$,
\begin{equation}
\mu^{-1}(0) =\{ p;\psi|_{\pi^{-1}\pi(p)}\ 
\text{attains\  its\  minimum\ at\ } p \} .
\end{equation}
This observation implies that 
the restriction
$\pi|_{\mu^{-1}(0)}$ of $\pi$
to $\mu^{-1}(0)$
is surjective
and induces
a continuous bijective map $\mu^{-1}(0)/K \to Q$.
See, e.g.,
the proof of \cite[Section 3 Proposition 3.1.5 p.~329]{MR1748608}.
Since $\psi$ is bounded from below
and $\psi \times \pi$ proper,
the  map $\mu^{-1}(0)/K \to Q$
is a homeomorphism
\cite[Lemma 1 p.~131]{MR1293876},
\cite[Section 3 Proposition 3.1.7 p.~331]{MR1748608}.
\end{proof}

\begin{proof}[Proof of openness of $M^{\Rr}$ in $M$]
Let $q$ be a point of the zero locus $\mu_\RR^{-1}(0)$.
The proof of the slice theorem
\cite[\S 2.7 Theorem p.~292]{MR1274117}
yields an open slice neighborhood $N$ of $q$ in $M$ 
that underlies a
relative exhaustion
Stein $G$-manifold
$(N,G,K,\psi)$
in such a way
that $\psi$ determines the restrictions
to $N$ of
$\omega_\RR$ and $\mu_\RR$.
The construction of $\psi$ builds on
a similar construction
in \cite{MR1293876}
and in particular relies on
 \cite[Lemma~2]{MR1293876}.
\end{proof}

With these preparations out of the way, let
\begin{align*}
\mu^{-1}_\CC(0)^{\Rr}&=\mu^{-1}_\CC(0)\cap M^{\Rr},
\\
\mu &=(\mu_\RR,\mu_\CC) \colon M \longrightarrow \kk^* \times \gg^*.
\end{align*}
Then
\begin{equation}
M_0:=\mu^{-1}_\CC(0)^{\Rr}  //G
\label{anaquot}
\end{equation}
is the analytic quotient of $\mu_\CC^{-1}(0)$ we are looking for,
and
the inclusion 
$\mu_\RR^{-1}(0) \subseteq M^{\Rr}$
induces a homeomorphism
\begin{equation}
\mu^{-1}(0)/K
\longrightarrow
M_0=\mu^{-1}_\CC(0)^{\Rr}  //G.
\end{equation}
The orbit space
$ \mu^{-1}(0)/K$ is a topological model for the analytic quotient 
$M_0$
of 
$\mu_\CC^{-1}(0)$.

\subsection{Holomorphic local model}

\subsubsection{  $\TT^*\cgroup$}
Endow $\TT^*\cgroup$ with the algebraic cotangent bundle symplectic structure
and identify $\TT^*\cgroup$ biholomorphically with 
$\cgroup\times \gg^*$ via left translation.
Accordingly the action of $\cgroup\times \cgroup$ 
on $\TT^*\cgroup$ which left and right translation on $\cgroup$ induces
takes the form
\begin{equation}
\cgroup \times \cgroup \times \cgroup \times \gg^*
\longrightarrow
 \cgroup \times \gg^*,
\ 
(x,y,u, \xi) \longmapsto (xuy^{-1}, \Ad^*_y\xi),
\end{equation}
and the association 
\begin{equation}
\cgroup \times \gg^* \longrightarrow \gg^* \times \gg^*,\ 
(x, \xi) \longmapsto (\Ad^*_x\xi, -\xi).
\label{moma1}
\end{equation}
characterizes the algebraic $\cgroup \times \cgroup$-momentum mapping
that turns 
$\TT^*\cgroup \cong \cgroup\times \gg^*$
into a 
 $\cgroup \times \cgroup$-hamiltonian complex algebraic manifold
relative to the algebraic  cotangent bundle symplectic structure.

\subsubsection{Geometry of the local model}
\label{geometry}
Let $H$ be a reductive subgroup 
of $\cgroup$ and $V$ a complex symplectic  representation of $H$.
Write the complex symplectic form on $V$ as $\omega_V$.
The familiar algebraic momentum mapping 
\begin{equation}
\Phi_V\colon V \longrightarrow \hh^*,\ 
\Phi_V(v) (x) = \tfrac 12 \omega_V(xv,v),
\label{moma22}
\end{equation}
turns $V$ into a complex algebraic hamiltonian $H$-space.
Relative to the embedding of $H$ into $G \times G$
via the second copy of $G$,
take the product momentum mapping
\begin{equation}
\lambda\colon
G\times \gg^* \times V \longrightarrow \hh^*,
\ 
\lambda(x,\xi,v)= \Phi_V(v)- \xi|_\hh.
\end{equation}
Zero is a regular value of $\lambda$, the reduced space
$E = (\TT^*G \times V)//_\lambda H$
is a complex algebraic  manifold, acquires an algebraic  
symplectic
structure which we write as $\omega_E$ and, furthermore,
via the first copy of $G$, an algebraic 
hamiltonian  
$G$-action, with algebraic
momentum mapping coming from \eqref{moma1}.

Take a $K$-invariant 
hermitian inner product on $\gg$
and let $\mm$ be the orthogonal complement to $\hh$
in $\gg$. This identifies $\hh^*$ with the annihilator
$\mmo$ 
of $\mm$ in $\gg^*$
and $\mm^*$ with the annihilator $\hho$ 
of $\hh$ in $\gg^*$, 
and we thereby view $\Phi_V$
as taking values in $\gg^*$.
Then $E$ appears as  the total space of the algebraic 
vector bundle $E= G \times_H(\hho \times V) \to G/H$,
and the  algebraic momentum mapping reads
\begin{equation}
\kappa \colon G \times_ H(\hho \times V) \longrightarrow \gg^*,
\ 
[x,\xi,v] \longmapsto \Ad_x^*(\xi + \Phi_V(v)) .
\label{algmom}
\end{equation}
Furthermore, 
the zero section embedding $G/H \to E$ 
is isotropic relative to $\omega_E$.

The diagram
\begin{equation}
\begin{CD}
V @>>> G \times_ H(\hho \times V)
\\
@V{\Phi_V}VV
@VV{\kappa}V
\\
\mmo @>>> \gg^*
\end{CD}
\end{equation}
is commutative, and
the canonical injection $V \to  G \times_ H(\hho \times V)$
induces an isomorphism
\begin{equation}
\Phi_V^{-1}(0)//H \longrightarrow \kappa^{-1}(0)//G
\end{equation}
of algebraic GIT-quotients.

\subsubsection{Topology of the local model in 
terms of Kempf-Ness theory}
\label{topokn}
Write $V_0 = \Phi_V^{-1}(0)//H $ and let $\pi \colon  \Phi_V^{-1}(0) \to V_0$
denote the quotient map.
Let $\sigma_V$ be a (real) K\"ahler form on $V$
invariant under $L = H \cap K$, let $\ll$ denote the Lie algebra of $L$,
let
\begin{equation}
\mu_{\sigma_V} \colon V \longrightarrow\ll^*,\ 
x \circ \mu_{\sigma_V}(v) = \tfrac 12 \sigma_V(xv,v),\ 
v \in V,\  
\ll \ni x \colon\ll^* \to \RR,
\end{equation}
be the associated momentum mapping having the value zero at the origin,
and consider
\begin{equation}
\mu_V = (\mu_{\sigma_V},\Phi_V)\colon V \longrightarrow \ll^* \times \hh^*.
\end{equation}
The injection $\mu_V^{-1}(0) \subseteq \Phi_V^{-1}(0)$ induces a homeomorphism
$\mu_V^{-1}(0)/L \to  V_0=\Phi_V^{-1}(0)//H$.

Likewise the 
 injection $\mu_{\sigma_V}^{-1}(0) \subseteq V$ induces a homeomorphism
$\mu_{\sigma_V}^{-1}(0)/L \to  V//H$.
The left-hand side characterizes 
the topology and the right-hand side
the complex algebraic structure
of $V//H$, and the diagram
\begin{equation}
\begin{gathered}
\xymatrix{
  & \mu_{\sigma_V}^{-1}(0) \ar@{->}[rr]^{\subseteq}\ar@{->>}'[d][dd]
 & & V \ar@{->>}[dd]
\\
 \mu_V^{-1}(0) \ar@{->}[ur]^{\subseteq}
\ar@{->}[rr]_{\phantom{aaaaaaa}\subseteq}\ar@{->>}[dd]
 & &  \Phi_V^{-1}(0)\ar@{->}[ur]^{\subseteq}\ar@{->>}[dd]
\\
 & \mu_{\sigma_V}^{-1}(0)/L \ar@{->}'[r]_{\phantom {aaaa}\cong}[rr]
 & & 
V//H                                          
 \\            
 \mu_V^{-1}(0)/L \ar@{->}[rr]_{\cong}\ar@{>->}[ur]
 & &   V_0 \ar@{>->}[ur]
}
\end{gathered}
\label{topodiag}
\end{equation}
is commutative. By construction,
the domain of each inclusion written as $\subseteq$
and of each  injection  written as $\rightarrowtail$
carries the induced topology,  the range of each surjection
written as $\twoheadrightarrow$
carries the quotient topology, and the arrows labeled $\cong$
are homeomorphisms. 
Indeed, as for the topologies
of the spaces in the upper square the claim is immediate,
and the homeomorphisms
result form GIT.
Since the group $L$ is compact,
it is immediate that
$\mu_V^{-1}(0)/L$
carries the topology induced from
$\mu_{\sigma_V}^{-1}(0)/L$.
Hence $V_0$ carries the topology induced from
$V//H$.

\begin{prop}
\label{top}
A subset $U$ of $V_0=\Phi_V^{-1}(0)//H$ is open if and only if,
relative to the quotient map $\pi \colon \Phi_V^{-1}(0) \to V_0=\Phi_V^{-1}(0)//H$,
there is an $H$-saturated subset $W$ of $V$ 
such that $\pi^{-1}(V) =\Phi_V^{-1}(0) \cap W$.
\end{prop}

\begin{proof}
A subset $U$ of $V_0=\Phi_V^{-1}(0)//H$ is open if and only if
there is a open subset $W'$ of $V//H$ 
such that $U =V_0 \cap W'$. The pre-image of $W'$ in
$V$ is $H$-saturated. This implies the claim.
\end{proof}

\subsubsection{Variation of the hamiltonian structure of
 the local model}
Let $\eta_\CC$ 
be a $G$-holomorphic symplectic structure
on $E=G \times_ H(\hho \oplus V)$ 
with momentum mapping $\mu_\CC\colon E \to \gg^*$
and suppose that the zero section embedding
 $G/H \to E$ 
is isotropic relative to $\eta_\CC$.
Proposition 2 in
\cite[\S 3.2 p.~222]{MR2230091}, 
taken up in the proof of 
\cite[Theorem 1.3]{mayrand},
says the following.

\begin{prop}
\label{varia1}
There is a $G$-equivariant biholomorphism
$\chi \colon E \to E$ such that
$\chi^*(\eta_\CC)$ and $\omega_\CC$ coincide on
the image $Z$ of the zero section embedding $\GG/H \to E$. \qed
\end{prop}

The holomorphic extension 
\cite[\S 3.3 p.~223]{MR2230091},  \cite[Section 3]{mayrand} 
of the Darboux-Weinnstein theorem
\cite[Theorem 4.1, Corollary 4.3]{MR0286137},
reproduced 
in \cite[Theorem 22.1]{MR770935}, \cite[Theorem 6]{MR1486529},
\cite[7.3.1 Theorem]{MR2021152},  
implies the following.

\begin{prop}
\label{varia2}
Suppose that the restrictions of $\eta_\CC$
and $\omega_\CC$ to  the image $Z$ of the zero section embedding
$G/H \to E$ coincide.
Then there are  open $G$-invariant neighborhoods 
$U_0$ and $U_1$
of $Z$ in $E$ and a $G$-equivariant
biholomorphism $\vartheta \colon U_0 \to U_1$ such that
$\vartheta^*(\eta_\CC) = \omega_{E}$
and $\vartheta|_Z = \Id_Z$. \qed
\end{prop}

\subsubsection{Affine complex structure on $V_0$} The affine coordinate ring
$\CC[V_0]$ 
of the algebraic GIT-quotient
$V_0=\Phi_V^{-1}(0)//H$ is the ring
$\CC[\Phi_V^{-1}(0)]^H$  of $H$-invariants in the affine coordinate ring
$\CC[\Phi_V^{-1}(0)]$ of the complex algebraic set $\Phi_V^{-1}(0)$ and, 
accordingly,
\begin{equation}
V_0 = \Hom_\Alg(\CC[\Phi_V^{-1}(0)]^H,\CC)
=\mathrm{Spec}(\CC[\Phi_V^{-1}(0)]^H).
\end{equation}
Thus a complex-valued function $f$ on $V_0$
 belongs to  $\CC[V_0]$
if and only if there exists a  function
$\widehat f$ in the affine coordinate ring $\CC[V]$ of $V$
 that renders a diagram of the kind
\begin{equation}
\begin{CD}
\Phi^{-1}_V(0) @>{\subseteq}>> V
\\
@V{\pi}VV
@VV{\widehat f}V
\\
V_0
@>>{f}>
\CC
\end{CD}
\end{equation}
commutative.
While the composite $f \circ \pi$ is $H$-invariant,
there is no reason for
$\widehat f$ to be $H$-invariant.
 
\subsubsection{Complex analytic structure on $V_0$} 
\label{compans}
The sheaf $\OO_{V_0}$
of germs of holomorphic functions on $V_0$ arises as follows:
Let $U$ be an open set in $V_0$; then $\pi^{-1}(U)$ is open in
$\Phi^{-1}_V(0)$,
that is, for some open set $U'$ in $V$,
the subset $\pi^{-1}(U)$ coincides with 
$\Phi^{-1}_V(0) \cap U'$;
 a complex-valued function $f$ on $U$
is holomorphic, i.e., belongs to  
 $\OO_{V_0}(U)$,
if and only if there exists a holomorphic function
$\widehat f$ on  $U'$ that renders a diagram of the kind
\begin{equation}
\begin{CD}
\pi^{-1}(U) @>{\subseteq}>> U'
\\
@V{\pi}VV
@VV{\widehat f}V
\\
U
@>>{f}>
\CC
\end{CD}
\label{diag2}
\end{equation}
commutative.
While the composite $f \circ \pi$ is $H$-invariant,
there is, at first, no reason for
$\widehat f$ to be $H$-invariant.

By Proposition \ref{top},
we can take $U'$ to be $H$-saturated, however.
Then rendering $\widehat f$
invariant under the maximal compact subgroup $L$ of $H$
yields an $H$-invariant extension: the function $f^\sharp$
which the indentity
\begin{equation}
 f^\sharp(v) = \int_K\widehat f(x v) dx
\end{equation}
characterizes
is an $H$-invariant 
holomorphic function
on $V$ rendering, with $f^\sharp$ substituted for
$\widehat f$, diagram  \eqref{diag2} commutative. 
This establishes the following:

\begin{prop}
\label{epi}
Under the circumstances of Proposition {\rm \ref{top}},
the canonical restriction morphism ${\OO_V(W)^H \to\OO_{V_0}(U)}$ 
of algebras is an epimorphism. \qed
\end{prop}

\subsubsection{Algebraic Poisson structure}
The complex symplectic form $\omega_V$ on $V$ determines
an $H$-invariant algebraic Poisson bracket 
$\pbra$ on $\CC[V]$
and hence 
algebraic Poisson bracket 
$\pbra$ on $\CC[V]^H =\CC[V//H]$.
Let $\delta_V \colon \hh \to \CC[V]$
denote the comomentum
which the momentum mapping \eqref{moma22} 
induces, and let $I_{\Phi_V}$ be the ideal in 
$\CC[V]$ which $\delta_V (\hh) \subseteq \CC[V]$
generates.
By construction, the vanishing ideal
$I({\Phi^{-1}_V(0)})$ of the algebraic set $\Phi^{-1}_V(0)$
is the radical 
$\sqrt{I_{\Phi_V}}$ of the ideal
$I_{\Phi_V}$ in $\CC[V]$.
The paper \cite{MR1031826} explores the situation
in the real setting in great detail.

We know from the theory of constrained systems
that the ideal $I_{\Phi_V}^H$
of $H$-invariants is a Poisson ideal in $\CC[V]^H$.
By a theorem in \cite{MR618321}, 
the radical of an ideal of polynomials
closed under Poisson bracket
is also closed under Poisson bracket.
The quotient algebra
 $\CC[V]^H /(I({\Phi^{-1}_V(0)}))^H$
therefore presumably
yields a Poisson algebra of Zariski-continuous functions
on $V_0$.
Details remain to be checked.
Proposition \ref{holopois} below implies
that the ideal $(I({\Phi^{-1}_V(0)}))^H$ is a Poisson ideal.

\subsubsection{Holomorphic Poisson structure}
The complex symplectic form $\omega_V$ on $V$ 
induces a holomorphic Poisson structure $\pbra_{V_0}$ 
on $\OO_{V_0}$ as follows:
Let $U$ be open in $V_0$.
By Proposition \ref{top}, there is an $H$-saturated
open set $W$ in $V$ such that $\pi^{-1}(U)= \Phi_V^{-1}(0)\cap W$.

\begin{prop}
\label{holopois}
The symplectic Poisson structure
$\pbra_W$
on $\OO_V(W)$ induces a
Poisson structure $\pbra_U$ on
$\OO_{V_0}(U)$.
\end{prop}

\begin{proof}
Up to a change of notation, we must show that
the symplectic Poisson structure
$\pbra_V$
on the ring
$\OO_V(V)$ of entire functions on $V$ induces a
Poisson structure $\pbra_{V_0}$ on the ring
$\OO_{V_0}(V_0)$ of holomorphic functions
on $V_0$.

Since the symplectic form $\omega_V$ on $V$ is $H$-invariant,
the symplectic Poisson structure $\pbra_V$
on $\OO_V(V)$ induces a Poisson structure
on the subalgebra $\OO_V(V)^H$
of $H$-invariants.
By Proposition \ref{epi}, the canonical restriction morphism
$\OO_V(V)^H \to \OO_{V_0}(V_0)$ of algebras is an epimorphism.
The argument for
\cite[Theorem 1 p.~35]{MR1123275}
shows that
the ideal 
of $H$-invariant functions in $\OO_V(V)$ that vanish on
$\Phi_V^{-1}(0)$ is a Poisson ideal
in $\OO_V(V)^H$. Since this is, perhaps, not entirely obvious, we 
reproduce the details in the present holomorphic setting:

Let $f$ and $h$ be $H$-invariant entire holomorphic functions
on $V$
 with
$h|_{\Phi_V^{-1}(0)}=0$ and let $q \in \Phi_V^{-1}(0)$.
Thus $\Phi_V(q) =0$.
Write the holomorphic hamiltonian vector field
of $f$ as $X_f$ and,
for $Y \in \hh$, let $Y_V$ denote the linear  holomorphic
(algebraic)
vector field on $V$ which $Y$ induces.
We must show that
\begin{equation}
\{f, h\}_V(q) = - (X_f h)(q) = 0.
\label{show}
\end{equation}
Since
$f$ is $H$-invariant,
\begin{align*}
\{Y \circ \Phi_V,h\}_V &= -Y_Vf =0,\quad
\text{for}\ Y \in \hh.
\end{align*}
Consequently, for $Y \in \hh$,
the algebraic function 
$Y \circ \Phi_V$ is constant along the holomorphic integral curves of
$X_f$.
Hence the holomorphic integral curve
$z \mapsto \varphi^f_z(q)$ of $X_f$ 
($z$ in a neighborhood of $0 \in \CC$)
lies in
$\Phi_V^{-1} \Phi_V(q)= \Phi_V^{-1}(0)$.
Differentiating with respect to the variable $z$ and evaluating at
$z=0$
we find
\eqref{show}.
\end{proof}
The following is an immediate consequence of Proposition \ref{holopois}.

\begin{thm}
\label{holopois2}
Let $(V,\omega_\CC)$
be a complex symplectic representation 
of a complex reductive Lie group $H$, let $\sigma_V$
be a (real) K\"ahler form on $V$ invariant under a maximal compact 
subgroup $L$ of $H$,
let
$\mu_{\sigma_V} \colon V \to \ll^*$ and $\Phi_V\colon V \to \hh^*$
denote the associated momentum mappings, and let
$V_0=(\mu_{\sigma_V}^{-1}(0)\cap\Phi_V^{-1}(0))/L\cong \Phi_V^{-1}(0)//H$,
endowed with the reduced complex analytic structure
$\OO_{V_0}$ discussed in $\S$ {\rm \ref{compans}}.
The holomorphic Poisson structure $\pbra_V$ 
on $\OO_V$ induces a
Poisson structure $\pbra_{V_0}$ on
$\OO_{V_0}$. \qed
\end{thm}

\subsubsection{Hyperk\"ahler case}
Let $V \cong \mathbb H^n$ ($n \geq 1$) be a quaternionic vector space.
Let $\mathsf I$, $\mathsf J$, $\mathsf K$
denote three complex structures
that behave like quaternions (generate the quaternion group of order eight) 
and
generate the quaternionic structure, and let $\hinn$ 
be a real hyperk\"ahler metric
(i.e., $\hinn$ renders $\mathsf I$, $\mathsf J$, $\mathsf K$ skew).
This turns $V$ into a hyperk\"ahler manifold
with K\"ahler forms 
$\omega_{\mathsf I}(\cdot,\cdot) =\langle \mathsf I \cdot,\cdot \rangle$,
$\omega_{\mathsf J}(\cdot,\cdot) =\langle \mathsf J \cdot,\cdot \rangle$, 
$\omega_{\mathsf K}(\cdot,\cdot) =\langle \mathsf K \cdot,\cdot \rangle$.
Define
\begin{align}
\Phi_{\mathsf I}\colon V \to \ll^*,\quad x\circ \Phi_{\mathsf I}(v) &= \tfrac 12 \omega_{\mathsf I}(xv,v), 
\\
\Phi_{\mathsf J}\colon V \to \ll^*,\quad x\circ \Phi_{\mathsf J}(v) &= \tfrac 12 \omega_{\mathsf J}(xv,v), 
\\
\Phi_{\mathsf K}\colon V \to \ll^*,\quad x\circ\Phi_{\mathsf K}(v)&= \tfrac 12 \omega_{\mathsf K}(xv,v).
\end{align}

\begin{thm}
\label{hyper1}
Let $L$ be a compact Lie group acting linearly on $V$
and preserving the linear hyperk\"ahler structure
$\hinn$, $\mathsf I$, $\mathsf J$, $\mathsf K$.
The three complex structures  $\mathsf I$, $\mathsf J$, $\mathsf K$
determine,
on the hyperk\"ahler quotient $V_0 =\Phi^{-1}(0)/L$
relative to the hyperk\"ahler momentum mapping
\begin{equation}
\Phi\colon V \longrightarrow \ll^* \otimes \RR^3,\ (x_1,x_2,x_3)\circ \Phi(v)=
\tfrac 12
(
\omega_{\mathsf I}(x_1v,v), 
\omega_{\mathsf J}(x_2v,v),
\omega_{\mathsf K}(x_3v,v)
),
\end{equation}
 three respective complex analytic structures $\OO_{\mathsf I}$,
$\OO_{\mathsf J}$, $\OO_{\mathsf K}$, and the corresponding K\"ahler forms
on $V$ determine three pairwise compatible real Poisson structures
$\pbra_{\mathsf I}$,
$\pbra_{\mathsf J}$, $\pbra_{\mathsf K}$
on, respectively
 $\OO_{\mathsf I}$,
$\OO_{\mathsf J}$, $\OO_{\mathsf K}$,
such that
\begin{equation}
\left(\OO_{\mathsf I},\pbra_{\mathsf J}+i\pbra_{\mathsf K}\right), \quad
\left(\OO_{\mathsf J},\pbra_{\mathsf K}+i\pbra_{\mathsf I}\right), \quad
\left(\OO_{\mathsf K},\pbra_{\mathsf I}+i\pbra_{\mathsf J}\right)
\label{holopoisson}
\end{equation}
are holomorphic Poisson structures on $V_0$.
These generate a sphere of  holomorphic Poisson structures on $V_0$.
\end{thm}
\begin{proof}
Relative to $(\mathsf I,\omega_{\mathsf I}, \omega_{\mathsf J}+i
\omega_{\mathsf K}, \Phi_{\mathsf J}+i
\Phi_{\mathsf K})$,
the  affine space which underlies
the vector space $V$ is an $L^\CC$-hamiltonian holomorphic symplectic
K\"ahler manifold.
Theorem \ref{holopois2}
yields the holomorphic Poisson structure
$\left(\OO_{\mathsf I},\pbra_{\mathsf J}+i\pbra_{\mathsf K}\right)$.
Now, repeat the argument with
$(\mathsf J,\omega_{\mathsf J}, \omega_{\mathsf K}+i
\omega_{\mathsf I}, \Phi_{\mathsf K}+i
\Phi_{\mathsf I})$ and
$(\mathsf K,\omega_{\mathsf K}, \omega_{\mathsf I}+i
\omega_{\mathsf J}, \Phi_{\mathsf I}+i
\Phi_{\mathsf J})$.
\end{proof}
\subsubsection{Stratification}

Return to the linear $H$-hamiltonian holomorphic symplectic K\"ahler
manifold $(V,\sigma_V,\omega_\CC,\mu_{\sigma_V},\Phi_V)$
studied earlier in this section.
The reasoning in \cite[\S 4.7]{mayrand} establishes the following.

\begin{thm}
\label{orbit}
Let $(V,\omega_\CC)$
be a complex symplectic representation 
of a complex reductive Lie group $H$, let $\sigma_V$
be a (real) K\"ahler form on $V$ invariant under a maximal compact 
subgroup $L$ of $H$,
and let
$\mu_{\sigma_V} \colon V \to \ll^*$ and $\Phi_V\colon V \to \hh^*$
denote the associated momentum mappings.
The orbit type decomposition of the quotient
$V_0=(\mu_{\sigma_V}^{-1}(0)\cap\Phi_V^{-1}(0))/L\cong \Phi_V^{-1}(0)//H$ 
is a complex Whitney stratification.
Hence the complex analytic structure 
$\OO_{V_0}$ on $V_0$ which  the complex structure of $V$ determines
and the holomorphic Poisson structure
$\pbra_{V_0}$ 
on $\OO_{V_0}$
which the complex symplectic structure $\omega_\CC$ on $V$ 
induces turns
$(V_0,\OO_{V_0},\pbra_{V_0})$ 
into a stratified holomorphic symplectic space. \qed
\end{thm}
(N.B. 
In the statement of the theorem, there is a single complex 
structure on $V$ under discussion.)

\begin{thm}
\label{hyper2}
Under the circumstances of Theorem {\rm \ref{hyper1}},
the three holomorphic Poisson structures
{\rm \eqref{holopoisson}}
on the hyperk\"ahler quotient $V_0= \Phi^{-1}(0)/L$
are compatible with the orbit type stratification
of $V_0$ and thereby yield a stratified hyperk\"ahler structure.
Moreover,  the orbit type 
 stratification of $V_0$ is a complex Whitney stratification
relative to each of $\OO_{\mathsf I}$,
$\OO_{\mathsf J}$, $\OO_{\mathsf K}$. \qed
\end{thm}

\begin{rema}
\label{realset}
{\rm
In the real setting, in \cite[1.11 Example]{MR1127479},
Sjamaar-Lerman recall that
\cite[Theorem 1 p.~35]{MR1123275}
yields the real stratified symplectic Poisson structure.
In \cite[3.1 Proposition]{MR1127479},
they establish the existence of this Poisson structure
by a pointwise reasoning involving the stratification.
In \cite[Subsection 4.8]{mayrand}, Mayrand
extends the pointwise reasoning for
\cite[3.1 Proposition]{MR1127479}
in terms of the corresponding stratification
to the complex analytic case
in the realm of hyperk\"ahler manifolds
to construct a Poisson bracket of the kind
$\pbra_{V_0}$ in  Theorem \ref{holopois2}.
The proof of Theorem \ref{holopois2} 
is independent of the stratification.
}
\end{rema}

\subsection{Local structure of the analytic quotient of the
hamiltonian 
holomorphic symplectic K\"ahler manifold
at the start
}
Return to the circumstances of Subsection \ref{hsr}.
Let $p$ be a point of  $\mu^{-1}(0)\subseteq M$. 
By an observation in
\cite[\S 2.2]{MR1274117},  the stabilizer $H_p$ of $p$ is reductive
and hence the complexification of a 
compact group.

The $G$-action on $M$ turns 
the complex symplectic vector space  $(\TT_pM,\omega_\CC)$
into a symplectic  $H_p$-representation,
the tangent space
$\gg p =\TT_p(G\cdot p) \subseteq \TT_pM$
to the $G$-orbit at $p$
is a subrepresentation, and so is
the skew-orthogonal complement 
$\gg p^{\omega_\CC} \subseteq \TT_pM$
of $\gg p$.
In view of the momentum property,
$\gg p^{\omega_\CC}= \ker (d\mu_\CC)$, 
and
$\gg p \subseteq \gg p^{\omega_\CC}$
as the annihilator of
the restriction of
$\omega_\CC$ to  $\gg p^{\omega_\CC}$.
Relative to 
the hermitian form associated with
$\omega_\RR$,
let $V_p$ denote the orthogonal complement
of 
$\gg p$ in $\gg p^{\omega_\CC}$, so that
\begin{equation}
\gg p^{\omega_\CC} =\gg p \oplus V_p 
\end{equation}
is a decomposition of $H_p$-representations.
Analogously to terminology in
\cite[Section 2]{MR1127479}, say 
$V_p$ is an {\em infinitesimal
holomorphic symplectic slice
at\/} $p$ {\em for the $G$-action on\/} $M$. 
In the terminology of
\cite[7.2.1 Definition p.~276]{MR2021152},
the complex vector space $V_p$ is a {\em symplectic normal
space\/} at $p$.
The holomorphic symplectic structure $\omega_\CC$
on $M$ induces a complex symplectic form $\omega_p$ on $V_p$,
and the stabilizer $H_p \subseteq G$
of the point $p$ of $M$ acts linearly and symplectically on $V_p$.

Write $E_p= G \times _{H_p} (\hh_p^{\mathrm o} \times V_p)$.
Combining
the holomorphic slice theorem \cite[\S 2.7 Theorem p.~292]{MR1274117},
\cite[Theorem 1.12 p.~100]{MR1314032} 
with a holomorphic version of the Darboux-Weinstein theorem 
\cite[Theorem 4.1, Corollary 4.3]{MR0286137},
reproduced 
in \cite[Theorem 22.1]{MR770935}, \cite[Theorem 6]{MR1486529},
\cite[7.3.1 Theorem]{MR2021152},
and with some
extra work,
in \cite{mayrand}, Mayrand manages to establish the 
holomorphic local normal form of the momentum mapping or,
equivalently, the
holomorphic symplectic slice theorem,
in the realm of hyperk\"ahler manifolds.
We extend this result as follows.

\begin{thm}[Holomorphic symplectic slice theorem]
\label{hololoc}
Let $G$ be a 
complex reductive Lie group,
$K$ a maximal compact subgroup, and let $(M,\omega_\CC,\mu_\CC)$
be a $G$-hamiltonian 
holomorphic symplectic manifold which carries  a
 $K$-hamiltonion K\"ahler structure
 $(M,\omega_\RR,\mu_\RR)$ as well.
For an arbitrary  point
 $p$ in $\mu_\RR^{-1}(0)\cap \mu_\CC^{-1}(0)$, 
by construction necessarily in 
$M^{\Rr}$, cf. {\rm \eqref{semistable}},
there is a $G$-saturated neighborhood $U_p$ of $p$ in 
$M^{\mu_\RR-\mathrm ss}$,
a $G$-saturated neighborhood $U_p'$ 
in $G\times_{H_p} (\hho_p \times V_p)$
of 
the image ($\cong G/H_p$) of the zero section of the vector bundle 
$E_p=G\times_{H_p} (\hho_p \times V_p)\to G/H_p$, and an
isomorphism
\begin{equation}
(U'_p,\omega_{E_p},\kappa_p)
\longrightarrow
(U_p,\omega_\CC,\mu_\CC)
\label{isoham}
\end{equation}
of $G$-hamiltonian complex manifolds
mapping the point $[e,0,0]$ 
of $G\times_{H_p} (\hho_p \times V_p)$
(which the point $(e,0,0)$ of $G\times \hho_p \times V_p$ represents)
to $p$.
\end{thm}

Since Mayrand merely proceeds in the hyperk\"ahler setting,
we explain the salient steps of the proof. 
The strategy of the proof is classical, see
 \cite[2.5. Proposition]{MR1127479}, \cite[Theorem 3]{MR2230091},
\cite[7.4.1 Theorem p.~282]{MR2021152},
\cite[7.5.5 Theorem p.~285]{MR2021152}.

\begin{proof}
Let $p$ be a point in $\mu_\RR^{-1}(0)\cap \mu_\CC^{-1}(0)$, and recall
$\gg p = \TT_p(G \cdot p) \cong \gg/\hh_p \cong \mm_p$, cf. 
$\S$ \ref{geometry} 
for the notation.
The symplectic polar $V_p^{\omega_\CC}\subseteq \TT_pM$
of $V_p$  is an $H_p$-re\-pre\-sentation
 and, for some Lagrangian complement $W_p$ of $\gg p$ in $V^{\omega_\CC}$,
necessarily an $H$-representation,
\begin{equation}
V_p^{\omega_\CC} = W_p \oplus \gg p  
\end{equation}
as $H_p$-representations in such a way that
the map
\begin{equation}
\vartheta \colon W_p \longrightarrow (\gg p)^*,\ 
(\vartheta(Y))(X) = \omega_\CC(X,Y),\ 
X \in \gg p,\ Y \in W_p
\end{equation}
is an $H_p$-linear isomorphism.
Thus the resulting decomposition
\begin{equation}
\TT_pM \cong \mm_p \oplus \mm_p^* \oplus V_p 
\cong \mm_p \oplus \hho_p \oplus V_p
\label{result1}
\end{equation}
is a complex {\em Witt-Artin\/}
decomposition (relative to the symplectic structure
and momentum mapping), cf., e.g., \cite[7.1.1 Theorem]{MR2021152}
for the real case;
with regard to the tautological symplectic structure on
$\mm_p \oplus \mm_p^*$, 
the decomposition \eqref{result1} 
is one of complex symplectic $H_p$-representations.

The $G$-orbit $G\cdot p \subseteq M$ of $p$ in $M$
is a complex submanifold of $M$  and,
in view of the holomorphic slice theorem
\cite[\S 2.7 Theorem p.~292]{MR1274117},
there is a locally closed $H_p$-invariant complex
submanifold $S_p$ such that
the canonical $G$-equivariant holomorphic map
\begin{equation}
G \times_{H_p} S_p \longrightarrow G S_p
\end{equation}
is a $G$-equivariant
biholomorphism onto the open $G$-invariant $G$-saturated neighborhood
$G S_p$ of the $G$-orbit $G\cdot p$ of $p$ in $M$.
To establish the claim,
it suffices to argue in terms of
$G \times_{H_p} S_p$, that is,
near the point $p$, we can take $M$ to be 
$G \times_{H_p} S_p$.

By construction, the injection $S_p \subseteq M$ induces,
via the decomposition \eqref{result1}, an isomorphism
$\TT_p S_p \to \hho \oplus V_p$ and,
in view of the decomposition \eqref{result1}, the complex vector space 
$W_p=\TT_pM/T_p(G\cdot p) \cong \hho_p \oplus V_p$
serves as an ordinary infinitesimal holomorphic 
(beware: not symplectic)  slice 
at $p$ for the $G$-action on $M$.
Hence parametrizing $S_p$ holomorphically by its tangent space 
$\hho \oplus V_p$ at $p$
yields a local biholomorphism
between
$E_p=G\times_{H_p}(\hho_p \oplus V_p)$ 
and
$G \times_{H_p} S_p$
near the point $p$,
 that is,
there is an open $G$-invariant
neighborhood $U$ of $p$ in $M^{\mu_\RR-ss}$, an
open $G$-invariant
neighborhood $U'$ 
in 
$G\times_{H_p}(\hho_p \oplus V_p)$ of 
the image $Z_p\cong G/H_p$ of the zero section of the vector bundle 
$G\times_{H_p}(\hho_p \oplus V_p) \to G/H_p$, and a $G$-equivariant
biholomorphism $\Psi\colon U' \to U$ 
mapping the point $[1,0,0]$ 
of $E_p=G\times_{H_p}(\hho_p \oplus V_p)$ which
the point $(1,0,0)$ 
of $G\times (\hho_p \oplus V_p)$
represents to $p$.
By \cite[Proposition 3.8]{mayrand}, every $G$-invariant neighborhood of $p$
contains a $G$-saturated neighborhood of $p$.
Hence we may take $U$ and $U'$ to be $G$-saturated   in 
$M^{\mu_\RR-ss}$. 

Now, the complex algebraic manifold 
$E_p=G\times_{H_p}(\hho_p \oplus V_p)$
carries the algebraic $G$-invariant symplectic structure
$\omega_{E_p}$ and
the  $G$-invariant 
holomorphic symplectic structure $\eta_\CC =\Psi^*(\omega_\CC)$,
and the zero section $G/H_p \to E_p$ is an isotropic embedding for both.
While the restrictions
of $\omega_{E_p}$ 
and $\eta_\CC$ to $G/H_p$ need not coincide,
Propositions \ref{varia1} and \ref{varia2}
imply that
there are  open $G$-invariant neighborhoods 
$U_0$ and $U_1$
of the image $Z$ of the zero section in $E$ and a $G$-equivariant
biholomorphism $\rho \colon U_0 \to U_1$ such that
$\rho^*(\eta_\CC) = \omega_{E_p}$, and we may take
$U_0$ and $U_1$ to be $G$-saturated.
Shrinking the open neighborhoods if need be
and combining $\rho$ and $\Psi$
yields the isomorphism \eqref{isoham}.
Compatibility with the momentum mappings is a consequence of the
fact that a momentum mapping is unique up to a constant value in the center.
\end{proof}

\begin{rema}
{\rm
For an algebraic hamiltonian action of a reductive group $G$
on a non-singular affine symplectic variety, \cite[Theorem 3]{MR2230091}---Losev calls it
a \lq\lq symplectic slice theorem\rq\rq---establishes an analytical equivalence  
at an arbitrary point having closed orbit
between a saturated neighborhood of that point and
a saturated neighborhood of the corresponding point of a model space
of the kind $G\times_H(\hho \oplus V)$. 
The reader will notice there is no auxiliary K\"ahler form 
of the kind $\omega_\RR$ (cf. Subsection \ref{hsr} above) or $\sigma_V$
(cf. \S 3.3.3 above)
present in
\cite[Theorem~3]{MR2230091}.
}
\end{rema}

\subsection{Globalization}

To spell out global versions of
Theorems \ref{holopois2}, \ref{hyper1}, \ref{orbit} and \ref{hyper2},
as before, let $G$ be a complex reductive Lie group and
$K$ a maximal compact subgroup.
\begin{rema}
\label{independent}
{\rm
It is important to note that, in Theorems \ref{orbit2},
\ref{hyper4}, and \ref{hyper3} below
the complex analytic and the  holomorphic Poisson structures
are independent of the stratifications 
(orbit type decompositions)
and, in fact,
can be understood independently
of the corresponding 
orbit type decomposition;
 in each case, the orbit type decomposition
being a
stratification is
an additional piece of structure.
}
\end{rema}

\begin{thm}
\label{orbit2}
Let $(M,\omega_\CC,\mu_\CC)$
be a $G$-hamiltonian 
holomorphic symplectic manifold endowed with, furthermore, a
 $K$-hamiltonion K\"ahler structure
 $(M,\omega_\RR,\mu_\RR)$.
Then the complex structure of $M$
determines, on the
reduced space 
\begin{equation}
M_0= (\mu_\RR^{-1}(0) \cap\mu_\CC^{-1}(0)) /K \cong \mu_\CC^{-1}(0)^{\mu_\RR-\mathrm ss}//G,
\end{equation}
cf. {\rm \eqref{anaquot}},
a complex analytic structure
$\OO_{M_0}$, and 
the holomorphic symplectic form $\omega_\CC$ induces
a holomorphic Poisson bracket $\pbra_{M_0}$
on $\OO_{M_0}$. Moreover,
the  orbit type decomposition of $M_0$ 
is a complex Whitney stratification and,
relative to that stratification,
$(M_0,\OO_{M_0},\pbra_{M_0})$ 
is a stratified holomorphic symplectic space.
\end{thm}

In terms of the notation
\begin{equation*}
\mu_M^{-1}(0) = \mu_\RR^{-1}(0) \cap\mu_\CC^{-1}(0),
M_{0,K} = \mu_\RR^{-1}(0)/K,
M_{0,G} =\mu_\CC^{-1}(0)^{\mu_\RR-\mathrm ss}//G,
\end{equation*}
diagram \eqref{topodiag}
globalizes to the commutative diagram:
\begin{equation}
\begin{gathered}
\xymatrix{
  & \mu_\RR^{-1}(0) \ar@{->}[rr]^{\subseteq}\ar@{->>}'[d][dd]
 & & \phantom{aaaa} M^{\mu_\RR-\mathrm ss} \ar@{->>}[dd]
\\
\mu_M^{-1}(0) \ar@{->}[ur]^{\subseteq}
\ar@{->}[rr]_{\phantom{aaaaaaa}\subseteq}\ar@{->>}[dd]
 & &  \mu_\CC^{-1}(0)^{\mu_\RR-\mathrm ss}\ar@{->}[ur]^{\subseteq}\ar@{->>}[dd]
\\
 & M_{0,K} \ar@{->}'[r]_{\phantom {aaaa}\cong}[rr]
 & & 
M//G                                          
 \\            
M_0 \ar@{->}[rr]_{\cong}\ar@{>->}[ur]
 & &   M_{0,G} \ar@{>->}[ur]
}
\end{gathered}
\label{topodiag4}
\end{equation}

\begin{proof}
This is a straightforward consequence of Theorem \ref{hololoc}.
Indeed, the constructions and arguments 
given there globalize in an obvious fashion.
That the stratification is a complex Whitney stratification
is due to Mayrand \cite[Theorem 1.4]{mayrand}.
His reasoning for the hyperk\"ahler case 
extends to the present more general case.
We leave the details to the reader.
\end{proof}

\begin{rema}
{\rm
In Theorem \ref{orbit2}, there is no piece of structure
on $M_0$ which the real K\"ahler form $\omega_\RR$ induces.
In the presence of more structure on $M$, 
Theorems \ref{hyper4} and \ref{hyper3}
show in particular that
$\omega_\RR$ then induces
a stratified K\"ahler structure on $M_0$.
}
\end{rema}

For the application
in Section \ref{twana}, Theorem \ref{orbit2} suffices.
However, the following results are
worth spelling out:
Let
$(M,\mathsf I, \mathsf J,\mathsf K,
\omega_{\mathsf I}, 
\omega_{\mathsf J},
\omega_{\mathsf K},
\mu_{\mathsf I},\mu_{\mathsf J}, 
\mu_{\mathsf K})$
be a $K$-trihamiltonian hyperk\"ahler manifold,
write $\mu^{-1}(0) =\mu_{\mathsf I}^{-1}(0)
\cap
\mu_{\mathsf J}^{-1}(0)
\cap 
\mu_{\mathsf K}^{-1}(0)
$,
and consider the hyperk\"ahler quotient 
$M_0= \mu^{-1}(0) /K$.
Let $C^\infty(M_0)$
denote the image, under restriction, of
$C^\infty(M)^K$
in the algebra of continuous functions on $M_0$.

\begin{thm}
\label{poisson}
The three K\"ahler forms 
$\omega_{\mathsf I}, 
\omega_{\mathsf J},
\omega_{\mathsf K}$
induce 
three Poisson structures $\pbra_{\mathsf I,0}$, 
$\pbra_{\mathsf J,0}$,  $\pbra_{\mathsf K,0}$
on $C^\infty (M_0)$
that constitute
a stratified Poisson hyperk\"ahler structure. 
\end{thm}

\begin{rema}
{\rm
In Theorem \ref{poisson},
the term \lq\lq stratified\rq\rq\ refers to a notion of stratification
in the sense of \cite{MR572580}, 
weaker than that of a Whitney stratification.
 }
\end{rema}

\begin{proof}
By \cite[Theorem 2.1]{MR1468352},
on each piece of the orbit type decomposition,
the three K\"ahler forms 
$\omega_{\mathsf I}, 
\omega_{\mathsf J},
\omega_{\mathsf K}$
induce a hyperk\"ahler structure.
On a stratum, let
$\mathsf I_0$, $\mathsf J_0$,
$\mathsf K_0$
denote the corresponding complex structures and
$\omega_{\mathsf I,0}, 
\omega_{\mathsf J,0},
\omega_{\mathsf K,0}$
the corresponding K\"ahler forms.

To construct the Poisson structures,
we adapt the
pointwise reasoning
in \cite[3.1 Proposition]{MR1127479},
cf. Remark \ref{realset}, to the present situation as follows:

Let $f,h$ be in $C^\infty(M_0)$ and
let $q$ be a point of $M_0$.
The point $q$ lies in a unique orbit type piece $S_q$ of the orbit 
type decomposition of
$M_0$,
a hyperk\"ahler manifold, so take
\begin{equation}
\{f,h\}_{\mathsf I,0}(q) = \{f,h\}_{\mathsf I,0,S_q}(q)\ 
(\omega_{\mathsf I,0}\text{-symplectic\ Poisson\ bracket\ in\ } 
C^\infty(S_q)).
\end{equation}
It then remains to prove that
$\{f,h\}_{\mathsf I,0}$ is a member of $C^\infty (M_0)$.

By construction, 
there are $K$-invariant smooth functions
$\widehat f$ and $\widehat h$ on $M$
rendering the diagrams 
\begin{equation*}
\begin{CD}
\mu^{-1}(0)
@>>> M
\\
@VVV
@VV{\widehat f}V
\\
M_0 @>>f> \RR
\end{CD}
\quad \quad
\begin{CD}
\mu^{-1}(0)
@>>> M
\\
@VVV
@VV{\widehat h}V
\\
M_0 @>>h> \RR
\end{CD}
\end{equation*}
commutative.

The ordinary $\omega_{\mathsf I}$-symplectic Poisson bracket
$\pbra_{\mathsf I}$ on $C^\infty(M)$ is $K$-invariant.
Hence the smooth function
$\{\widehat f, \widehat h\}_{\mathsf I}$ on $M$ is $K$-invariant.
This function renders the diagram
\begin{equation*}
\begin{CD}
\mu^{-1}(0)
@>>> M
\\
@VVV
@VV{\{\widehat f, \widehat h\}_{\mathsf I}}V
\\
M_0 @>>{\{f, h\}_{\mathsf I,0}}> \RR
\end{CD}
\end{equation*}
commutative.

Repeating the argument with
regard to $\mathsf J$ 
and $\mathsf K$ yields the Poisson brackets 
$\pbra_{\mathsf J,0}$ and  $\pbra_{\mathsf K,0}$, respectively.

By \cite[Theorem 1.2]{mayrand},
the orbit type decomposition  is a stratification
in the sense of \cite{MR572580}.
\end{proof}

\begin{rema}
{\rm
The reasoning in the above proof shows that
the ideal of $K$-invariant functions
in $C^{\infty}(M)^K$
which vanish on 
$\mu^{-1}(0)$ is  Poisson ideal in
$C^{\infty}(M)^K$ relative to each of the Poisson structures
$\pbra_{\mathsf I}$,
$\pbra_{\mathsf J}$,
$\pbra_{\mathsf K}$
 on $C^\infty(M)$ 
associated with, respectively,
$\omega_{\mathsf I}, 
\omega_{\mathsf J},
\omega_{\mathsf K}$.
It would be interesting to establish this fact
by extending the argument for
\cite[Theorem 1 p.~35]{MR1123275},
cf. Remark \ref{realset}.
}
\end{rema}

Combining Theorems \ref{hololoc},  \ref{orbit2} 
and \ref{poisson}
leads to the following.

\begin{thm}
\label{hyper4}
Suppose that the
$K$-action integrates to a holomorphic
$G$-action relative to $\mathsf I$.
Then $\omega_{\mathsf I}$ induces
a stratified K\"ahler structure $(C^\infty(M_0), \OO_{M_0},\pbra_\RR)$ on
$M_0$,
cf. Proposition {\rm \ref{holopois}} and Remark {\rm \ref{realset}}.
Furthermore, this stratified K\"ahler structure combines with
the stratified holomorphic symplectic structure
$(\OO_{M_0},\pbra_{M_0})$ which $\mathsf I$
and $\omega_{\mathsf J} +i\omega_{\mathsf K}$ together with
$\omega_{\mathsf I}$,
in view of Theorem {\rm \ref{orbit2}}, determine,
  to
a weak stratified hyperk\"ahler structure
on $M_0$ relative to
the orbit type  stratification of $M_0$.
Moreover, this stratification is a complex  Whitney stratification
relative to $\OO_{\mathsf I}$.
Finally,
on the local model in Theorem {\rm \ref{hololoc}},
more precisely,
on the left-hand side $(U'_p,\omega_{E_p},\kappa_p)$ of {\rm \eqref{isoham}},
the other pieces of structure
$\mathsf J$, $\mathsf K$ 
and $\omega_{\mathsf I}$
on $M$ induce
not necessarily linear
complex structures and a K\"ahler form
that turn the local model
into a $\group$-trihamiltonian hyperk\"ahler manifold. \qed
\end{thm}

Repeating the reasoning for
Theorem \ref{hyper4}
with regard to the complex structures 
$\mathsf J$ and
$\mathsf K$
leads to the following, which is
\cite[Corollary 3.1.10]{maxence2019a}.

\begin{thm}
\label{hyper3}
Suppose that the 
$K$-action integrates to  holomorphic
$K^{\mathbb C}$-actions relative to each of
$\mathsf I$, $\mathsf J$ and $\mathsf K$.
Then the hyperk\"ahler structure 
induces
 three complex analytic structures $\OO_{\mathsf I}$,
$\OO_{\mathsf J}$, $\OO_{\mathsf K}$
and three pairwise compatible real Poisson structures
$\pbra_{\mathsf I}$,
$\pbra_{\mathsf J}$, $\pbra_{\mathsf K}$
on $M_0$
such that
\begin{equation}
\left(\OO_{\mathsf I},\pbra_{\mathsf J}+i\pbra_{\mathsf K}\right),\  
\left(\OO_{\mathsf J},\pbra_{\mathsf K}+i\pbra_{\mathsf I}\right),\  
\left(\OO_{\mathsf K},\pbra_{\mathsf I}+i\pbra_{\mathsf J}\right)
\label{holopoissong}
\end{equation}
are holomorphic Poisson structures.
These generate a sphere of  holomorphic Poisson structures on $M_0$.
Moreover,  the orbit type 
decomposition of $M_0$ is a complex Whitney stratification
relative to each of $\OO_{\mathsf I}$,
$\OO_{\mathsf J}$, $\OO_{\mathsf K}$, and the 
three holomorphic Poisson structures constitute
 a stratified hyperk\"ahler structure on $M_0$. \qed
\end{thm}

\begin{rema}
{\rm
Theorem \ref{orbit2} together with Theorem \ref{hololoc} extends, in a sense,
 \cite[Theorem 1.4]{mayrand} given there in the realm of 
hyperk\"ahler manifolds
to holomorphic symplectic K\"ahler manifolds
but offers a weaker conclusion, however:
The strata in Theorem \ref{orbit2} are holomorphic
symplectic manifolds
whereas those in 
 \cite[Theorem 1.4]{mayrand} are hyperk\"ahler.
Theorem \ref{hyper4} together with Theorem \ref{hololoc} 
essentially recovers
 \cite[Theorem 1.4]{mayrand}.
}
\end{rema}

\begin{rema}
{\rm
For $Y \in \lieal$, let $Y_M$ denote the induced smooth vector field
on $M$.
The $\group$-action on $M$ is integrable for the complex structure $\mathsf I$,
i.e., extends to a holomorphic $\cgroup$-action on $M$,
if and only if, for $Y \in \lieal$, the smooth vector field $\mathsf I\, Y_M$
on $M$ is complete.
Thus,
for $M$ compact, the $K$-action is integrable for any complex structure, 
cf., e.g., \cite[Theorem 4.4]{MR664118},
and the 
three holomorphic Poisson structures in Theorem \ref{hyper3} constitute
 a stratified hyperk\"ahler structure on $M_0$.
}
\end{rema}

\section{Twisted algebraic representation varieties}
\label{tar}
Retain the notation of Section \ref{grcoho}.
The $\cgroup$-subspace  $\Hom(\pi,\cgroup)$ of 
$\Hom(F,\cgroup) \cong \cgroup^{2\ell}$
is Zariski-closed and hence an affine $\cgroup$-variety.
By definition, the affine categorical quotient 
$\Hom(\pi,\cgroup)//\cgroup$
is the affine variety having
$\CC[\Hom(\pi,\cgroup)]^\cgroup$
as its coordinate ring,  and
we take the {\em algebraic representation variety\/}
$\Rep_\alg(\pi, \cgroup)$
associated with $\pi$ and $\cgroup$ 
to be this quotient; cf., e.g.,
\cite[Proposition 6.1 p.~11]{MR1320603}.
By general principles, the projection
$\pi\colon \Hom(\pi,\cgroup) \to \Rep_\alg(\pi, \cgroup)$
is a $\cgroup$-reduction in the sense of
Subsection \ref{quotients}, cf., e.g., \cite[Section 3]{MR1044583}.
The closed $\cgroup$-orbits are the semisimple representations,
the quotient
$\Rep_\alg(\pi, \cgroup)$ parametrizes the 
closed $\cgroup$-orbits, i.e., the semisimple representations,
and
each $\cgroup$-orbit in  $\Hom(\pi,\cgroup)$ 
has its {\em semisimplification\/}
as the unique closed $\cgroup$-orbit
in its closure \cite{MR233371}.
The injection
$\Hom^{\ssi}(\pi,\cgroup)\subseteq 
\Hom(\pi,\cgroup)$
induces a homeomorphism
\begin{equation}
 \Hom^{\ssi}(\pi,\cgroup)/ \cgroup\longrightarrow \Rep_\alg(\pi, \cgroup)
\label{homeo1}
\end{equation}
from the space of $\cgroup$-orbits
in the subspace $\Hom^{\ssi}(\pi,\cgroup)$
of
semisimple representations in 
$\Hom(\pi,\cgroup)$ onto $\Rep_\alg(\pi, \cgroup)$.
These facts hold for both the Zariski and the classical (metric) topology.
In the terminology of
\cite {MR1307297, MR1320603}, $\Rep_\alg(\pi, \cgroup)$ is the ordinary Betti moduli space;
in \cite[Section 6 p.~11/12]{MR1320603}, Simpson proceeds 
more generally for the fundamental group of a K\"ahler manifold
but this need not concern us here.
The {\em nonabelian Hodge theorem\/}
establishes, among others, for $\cgroup = \mathrm{GL}(r,\CC)$,
 a homeomorphism
between the moduli space of semistable 
rank $r$
topologically trivial Higgs bundles
and $\Rep_\alg(\pi, \cgroup)$ over the surface $\Sigma$.
This goes back to
 \cite{MR887284} for the case of rank two Higgs bundles
and to
\cite {MR1307297, MR1320603} for the general case
(in particular for the fundamental group of an arbitrary K\"ahler manifold).
Suitably rephrased, this correspondence extends to arbitrary
complex reductive Lie groups of the kind $\cgroup$ under discussion.

A classical topological construction 
provides the means to recover the case of topologically non-trivial
bundles. Atiyah-Bott discuss this in detail for connected $K$
\cite[Section 6]{MR702806}; 
see also
\cite[Section 5]{MR1370113}, \cite[Section 3]{MR3836789}: 
Let $N$ denote the normal closure
of $r$ in $F$. Consider the quotient group $\Gamma = F\slash [F,N]$.
The image $[r] \in \Gamma$ of $r$ 
generates the central subgroup 
$\mathbb Z\langle [r]\rangle =N\slash [F,N]$ of $\Gamma$,
 and
the resulting extension
\begin{equation}
\begin{CD}
0
@>>>
\mathbb Z\langle [r]\rangle
@>>>
\Gamma
@>>>
\pi
@>>>
1
\end{CD}
\label{3.2}
\end{equation}
is central.
Since the transgression homomorphism
$\Ho_2(\pi) \to \mathbb Z\langle [r]\rangle$
is an isomorphism, the extension \eqref{3.2}
is a maximal stem extension (Schur cover) and since,
furthermore, the abelianization of $\pi$ 
is a free abelian group,
that maximal stem extension is unique
up to within isomorphism \cite[\S 9.9 Theorem 5 p.~214]{MR0279200}.
Atiyah and Bott use the terminology 
\lq\lq universal central extension\rq\rq\ 
to refer to this situation
\cite[\S 6]{MR702806}.

Let $X$ be a member of the center $\zz$ of $\lieal$
such  that
$\exp(X)$ lies in the center of $K$.
When $K$ is connected, 
$\exp(X)$ lies in the center of $K$
for any $X\in \zz$.
The canonical surjection $F \to \Gamma$
induces an injection  
$\Hom(\Gamma,\cgroup) \subseteq \Hom(F,\cgroup)$,
and this injection identifies 
a certain subspace
 of $ \Hom(\Gamma,\cgroup)$ with 
the subspace $r^{-1}(\mathrm{exp}(X))$, if non-empty,
of  $\Hom(F,\cgroup)$.
Thus, suppose  $r^{-1}(\mathrm{exp}(X))$ non-empty.
We then denote that subspace of $ \Hom(\Gamma,\cgroup)$
by $\Hom_{X}(\Gamma,\cgroup)$.
The member $X$ of the center $\zz$
recovers a topological characteristic class of a corresponding bundle.
See 
\cite[\S 6]{MR702806},
\cite[Proposition 3.1]{MR3836789} for details.
We take the {\em twisted algebraic 
representation variety $\Rep_{X,\alg}(\pi, \cgroup)$ associated to\/}
$X \in \zz$
to be the corresponding affine categorical quotient.
The homeomorphism \eqref{homeo1} generalizes to a homeomorphism
\begin{equation}
\Hom^{\ssi}_X(\Gamma,\cgroup)// \cgroup
\longrightarrow
\Rep_{X,\alg}(\pi, \cgroup) . 
\label{homeo2}
\end{equation}
The nonabelian Hodge correspondence extends to that case
and recovers all topological types of Higgs bundles on $\Sigma$.

\section{Twisted analytic representation varieties as
stratified holomorphic symplectic spaces}
\label{twana}

Let
 $\group$ be a maximal compact subgroup 
so that $\cgroup$ is the complexification $\group^{\mathbb C}$ 
of $\group$.
Endow the Lie algebra $\lieal$ of $\group$ with an invariant
inner product. Left trivialization, the polar decomposition of
$\cgroup=\group^{\mathbb C}$ and the inner product on $\lieal$ induce a
diffeomorphism
\begin{equation}
\TT^*\group \stackrel{\cong} \longrightarrow\TT\group \longrightarrow  \group
\times \lieal \longrightarrow \group^{\mathbb C} = \cgroup
\label{polar}
\end{equation}
compatible with $K$-left and right translation
in such a way that the complex structure on $\group^{\mathbb C}$
and the cotangent bundle symplectic structure on $\mathrm
T^*\group$ combine to a $\group$-bi-invariant K\"ahler structure
on $\cgroup$.
In \cite{kronheimer}, Kronheimer claims this without proof;
a proof is in \cite{MR1892462}.
Moreover, the cotangent bundle
momentum mappings for left and right translation
combine, relative to the K\"ahler form, to a momentum mapping 
$\mu_{\mathrm{cot}}\colon \cgroup \to \lieal^*$
for the $\group$-action on $\cgroup$ by conjugation in
$\cgroup$, and
 the inner product on $\lieal$
induces a non-degenerate $\CC$-valued invariant 
symmetric bilinear
form $\cdot$ on $\clieal$.
Taking the product structure, we obtain a
K\"ahler form $\omega_\RR$ on $\cgroup^{2 \ell}$
and, relative to the
diagonal $\group$-action
on $\cgroup^{2 \ell}$, the action on each copy of $\cgroup$
being by conjugation in
$\cgroup$, 
 a $\group$-momentum mapping
$\mu_\RR \colon \cgroup^{2 \ell} \longrightarrow \lieal^*$.
Thus $(\cgroup^{2\ell},\omega_\RR, \mu_\RR)$
is a $\group$-hamiltonian K\"ahler manifold.

In diagram \eqref{PB},
substitute, for the open $\cgroup$-invariant subset $O$ of $\clieal$,
an open neighborhood in $\clieal$ of $0 \in \clieal$ where the
exponential map is a biholomorphism onto an open neighborhood of 
the neutral element
$e$ of $\cgroup$. Then the restriction
$\eta \colon \mathcal M(\mathcal P,\cgroup) \to \cgroup^{2\ell}$
is a biholomorphism
onto a $\cgroup$-invariant open neighborhood of 
$\Hom(\pi,\cgroup) = r^{-1}(e) \subseteq  \cgroup^{2\ell}$,
and the $\group$-hamiltonian K\"ahler structure on
$(\cgroup^{2\ell},\omega_\RR, \mu_\RR)$
induces a 
 $\group$-hamiltonian K\"ahler structure 
$(\omega_\RR, \mu_\RR)$ on
$\mathcal M(\mathcal P,\cgroup)$;
here we slightly abuse the notation
$\omega_\RR$ and $\mu_\RR$.
Let $\omega_\CC = \omega_{c,\mathcal P}$, cf. \eqref{closed},
and 
$\mu_\CC= \mu_{c,\mathcal P}$, cf. \eqref{eq},
 and, in terms of the notation and terminology in Subsection \ref{hsr},
define the {\em analytic representation variety\/}
associated with $\pi$ and $\cgroup$ to be
the holomorphic symplectic quotient
\begin{equation}
\Rep_\an(\pi, \cgroup)=\mathcal M(\mathcal P,\cgroup)//_{\mu_\CC} \cgroup \cong
\mu^{-1}(0)/K.
\end{equation}
Likewise, let $X$ be a member of  the center $\zz$ 
of $\lieal$ such that
$\mu^{-1}_\CC(X)$ is non-empty. 
In diagram \eqref{PB},
substitute, for the open $\cgroup$-invariant subset $O$ of $\clieal$,
an open neighborhood in $\clieal$ of $X \in \clieal$ where the
exponential map is a biholomorphism onto an open neighborhood $\widehat O$ of 
$\exp(X) \in \cgroup$. As before, the restriction
$\eta \colon \mathcal M(\mathcal P,\cgroup) \to \cgroup^{2\ell}$
is a biholomorphism
onto a $\cgroup$-invariant open neighborhood 
$\widehat  {\mathcal M}(\mathcal P,\cgroup)$ in $\cgroup^{2\ell}$
of 
$\Hom_X(\Gamma,\cgroup) = r^{-1}(\exp(X) ) \subseteq  \cgroup^{2\ell}$,
and the $\group$-hamiltonian K\"ahler structure on
$(\cgroup^{2\ell},\omega_\RR, \mu_\RR)$
induces a 
 $\group$-hamiltonian K\"ahler structure 
$(\omega_\RR, \mu_\RR)$ on
$\mathcal M(\mathcal P,\cgroup)$;
here again we slightly abuse the notation
$\omega_\RR$ and $\mu_\RR$.
As before, let $\omega_\CC = \omega_{c,\mathcal P}$, cf. \eqref{closed},
and 
$\mu_\CC= \mu_{c,\mathcal P}$, cf. \eqref{eq},
in terms of the notation and terminology in Subsection \ref{hsr},
let
\begin{align*}
\mu^{-1}_\CC(X)^{\Rr}&=
\mu^{-1}_\CC(X)\cap M^{\Rr},
\end{align*}
and 
define the {\em twisted 
analytic representation variety\/}
associated with $\pi$, $\cgroup$, and $X$ to be
the holomorphic symplectic quotient
\begin{equation}
\Rep_{X,\an}(\pi, \cgroup)=
\mu^{-1}_\CC(X)^{\Rr}
// \cgroup \cong (\mu_\RR^{-1}(0) \cap\mu_\CC^{-1}(X))/K.
\end{equation}
Then 
$\Rep_{0,\an}(\pi, \cgroup)=\Rep_\an(\pi, \cgroup)$.
Theorem \ref{orbit2} implies the following.

\begin{thm}
\label{main}
The complex Lie group $\cgroup$, 
the invariant inner product on $\lieal$,
and the choice of
$X\in \zz$ 
 such that
$\exp(X)$ lies in the center of $K$
and
$\mu^{-1}_\CC(X)$ is non-empty
determine a stratified holomorphic symplectic structure
on the twisted analytic representation variety $\Rep_{X,\an}(\pi, \cgroup)$.
The stratification is a complex Whitney stratification.
\qed
\end{thm}

\begin{rema}
{\rm
The  stratified
holomorphic symplectic structure 
is independent of any complex structure on $\Sigma$.
}
\end{rema}

\begin{rema}
{\rm
Let $\varphi \colon \Gamma \to G$ be a representation which lies in
$\mu_\RR^{-1}(0) \cap\mu_\CC^{-1}(X)$.
Then $\varphi$ determines
a $\pi$-module structure on $\gg$, and we denote this $\pi$-module
by $\gg_\varphi$.
One can show that right translation identifies 
an infinitesimal holomorphic symplectic slice
at $\varphi$ as a point of $\mathcal M(\mathcal P,\cgroup)$
with $\Ho^1(\pi,\gg_\varphi) \cong \Ho^1(\Sigma,\gg_\varphi)$.
In particular, at a regular point $[\varphi]$
of $\Rep_{X,\an}(\pi, \cgroup)$,
a choice of representative $\varphi$ in $[\varphi]$
induces an isomorphism from 
$\Ho^1(\Sigma,\gg_\varphi)$ to the tangent space
$\TT_{[\varphi]}(\Rep_{X,\an}(\pi, \cgroup))$
to $\Rep_{X,\an}(\pi, \cgroup)$ at the point $[\varphi]$.
This kind of observation goes back to
\cite{MR0169956}.

Let $[\pi]\in \Ho_2(\pi,\ZZ) \cong \ZZ$
denote a fundamental class (generator).
Consider a general point $\varphi$ in
$\mu_\RR^{-1}(0) \cap\mu_\CC^{-1}(X)$.
The stabilizer $H_\varphi \subseteq \cgroup$ acts linearly on
$\Ho^1(\pi,\gg_\varphi)$,
the pairing
\begin{equation}
\omega_\varphi\colon \Ho^1(\pi,\gg_\varphi) \otimes \Ho^1(\pi,\gg_\varphi)
\stackrel{\cdot\,\circ\, \cup} \longrightarrow \Ho^2(\pi,\CC)
\stackrel {\cap [\pi]}\longrightarrow \CC
\end{equation}
is skew and, in view of Poincar\'e duality, nondegenerate, i.e.,
a symplectic structure, necessarily $H_\varphi$-invariant.
Moreover, 
\begin{equation}
\mu_\varphi\colon \Ho^1(\pi,\gg_\varphi)
\stackrel{\cup \circ \bra} \longrightarrow \Ho^2(\pi,\gg_\varphi)
\stackrel {\cong}\longrightarrow \hh_\varphi^*
\end{equation}
recovers the associated momentum mapping having the value zero at the origin
\cite{MR762512}.
The resulting symplectic quotient 
$\Ho^1(\pi,\gg_\varphi)//H_\varphi$
is a local model for 
$\Rep_{X,\an}(\pi, \cgroup)$
near $[\varphi]$ as a stratified holomorphic symplectic space.
We can interpret this by saying that
$\Ho^1(\pi,\gg_\varphi)//H_\varphi$ 
yields generalized analytic 
Darboux coordinates for $\Rep_{X,\an}(\pi, \cgroup)$
near $[\varphi]$.
In particular, at a regular point  $[\varphi]$, this yields
ordinary holomorphic Darboux coordinates for $\Rep_{X,\an}(\pi, \cgroup)$
near $[\varphi]$.
See \cite{MR1938554} and the literature there for the corresponding
spaces of representations in a real Lie group.

Endow the (real) surface  $\Sigma$ with  a complex structure.
Using the corresponding
Hodge decomposition,
one can put a complex structure on 
$\Ho^1(\Sigma,\gg_\varphi)$
distinct from that coming from the complex structure of $\gg$.
As $\varphi$ varies, 
one can, perhaps, in this way recover the requisite complex
analytic  and  holomorphic Poisson structures and
prove that an analytic representation variety 
of the kind
$\Rep_{X,\an}(\pi, \cgroup)$
acquires a stratified hyperk\"ahler structure
which in particular recovers the  hyperk\"ahler structure
on the top stratum built in \cite{MR887284}.
}
\end{rema}

\section{Comparison of the twisted 
analytic and algebraic representation varieties}

\begin{thm}
Let $X$ be a member of  the center $\zz$ 
of $\lieal$ such that
$\exp(X)$ lies in the center of $K$ and that
$\mu^{-1}_\CC(X)$ is non-empty.
The holomorphic map 
$\eta \colon \mathcal M(\mathcal P,\cgroup) \to \cgroup^{2\ell}$
induces an analytic isomorphism
\begin{equation}
\Rep_{X,\an}(\pi, \cgroup) \longrightarrow \Rep_{X,\alg}(\pi, \cgroup).
\end{equation}
\end{thm}

Thus the  twisted analytic representation varieties
recover the Betti moduli spaces as
analytic spaces.

\begin{proof}
The diagram
\begin{equation}
\begin{CD}
\mu^{-1}_\CC(X)
@>{\subseteq}>>
\mathcal M(\mathcal P,\cgroup)
@>r_O>> O 
\\
@V{\eta_|}VV
@V{\eta}VV
@VV{\mathrm{exp}}V
\\
\Hom_{X}(\Gamma,\cgroup)
@>{\subseteq}>>
\widehat {\mathcal M}(\mathcal P,\cgroup)
@>>r> \widehat O 
\end{CD}
\end{equation}
is commutative, $\exp \colon O \to \widehat O$
and $\eta$ are biholomorphisms,
and  $\eta$  restricts to an isomorphism
$\eta_|\colon \mu^{-1}_\CC(X) \to \Hom_{X}(\Gamma,\cgroup)$
of analytic sets.

Consider the momentum mapping $\mu_\RR \colon \cgroup^{2\ell} \to \lieal^*$.
The zero locus $\mu_\RR^{-1}(0) \subseteq \cgroup^{2\ell}$
is a {\em Kempf-Ness\/} set (in the algebraic sense)
for the algebraic $\cgroup$-action 
on $\cgroup^{2\ell}$,
and
$\mu_\RR^{-1}(0)\cap \Hom_{X}(\Gamma,\cgroup)$
is a Kempf-Ness set 
for the algebraic $\cgroup$-action 
on $\Hom_{X}(\Gamma,\cgroup)$.
Hence the injection of
${\mu_\RR^{-1}(0)\cap \Hom_{X}(\Gamma,\cgroup)}$ into $\Hom_{X}(\Gamma,\cgroup)$
induces a homeomorphism 
\begin{equation}
(\mu_\RR^{-1}(0)\cap \Hom_{X}(\Gamma,\cgroup))/\group
 \longrightarrow \Hom_{X}(\Gamma,\cgroup)//\cgroup =
\Rep_{X,\alg}(\pi, \cgroup) .
\end{equation}
See, e.g., \cite{MR1040861} for details.

Likewise, relative to the momentum mapping
$\mathcal M(\mathcal P,\cgroup) \stackrel{\eta} \longrightarrow 
 \cgroup^{2\ell} \stackrel{\mu_\RR}\longrightarrow \lieal^*$---above
we also used the notation $\mu_\RR$ for it---,
the zero locus 
$(\mu_\RR \circ \eta)^{-1}(0) \subseteq \mathcal M(\mathcal P,\cgroup)$
is a {\em Kempf-Ness\/} set (in the analytic sense)
for the analytic $\cgroup$-action 
on $\mathcal M(\mathcal P,\cgroup)$,
and
$(\mu_\RR \circ \eta)^{-1}(0) \cap \mu^{-1}_\CC(X)$
is a Kempf-Ness set 
for the analytic $\cgroup$-action 
on $\mu^{-1}_\CC(X)$.
See \cite[\S 1.2]{MR1274117}
for this notion of Kempf-Ness set.
By 
\cite[Intro \S 1.3  p.~289, \S 3.3 Theorem p.~295]{MR1274117},
the injection
\begin{equation}
(\mu_\RR \circ \eta)^{-1}(0) \cap \mu^{-1}_\CC(X)
\longrightarrow \mu^{-1}_\CC(X)
\end{equation}
induces a homeomorphism 
\begin{equation}
((\mu_\RR \circ \eta)^{-1}(0) \cap \mu^{-1}_\CC(X))/\group
 \longrightarrow 
\mu^{-1}_\CC(X)//\cgroup =
\Rep_{X,\an}(\pi, \cgroup) .
\end{equation}
However, $\eta$ also induces a homeomorphism
\begin{equation}
((\mu_\RR \circ \eta)^{-1}(0) \cap \mu^{-1}_\CC(X))/\group
 \longrightarrow 
(\mu_\RR^{-1}(0)\cap \Hom_{X}(\Gamma,\cgroup))/\group.
\end{equation}
This implies the claim.
\end{proof}

\section*{Appendix I: Self-duality equations}

The self-duality equations first arose on the complex plane and 
constructing many solutions was tricky \cite{MR737994, MR456131, MR761842}
until Hitchin managed to avoid the tricky problems with boundary conditions 
by putting the equations on a compact surface, thereby getting nice moduli spaces
\cite{MR887284}.
However solving these tricky problems with boundary conditions led  to interesting moduli
spaces, forming complete hyperk\"ahler manifolds, even in the original case of $\RR^2$. 
The reader can find more details in the reviews
\cite{MR3931781, wildramification}. 
In particular the complex representation varieties in 
the present paper admit upgradings  to 
corresponding \lq\lq wild representation
varieties\rq\rq\ 
(\lq\lq wild character varieties\rq\rq\  in the terminology of \cite{MR3931781}).

\section*{Appendix II}
We profit from the opportunity to correct a minor technical flaw in
\cite{MR1460627}.
We are indebted to Suhyoung Choi for having isolated this flaw.

The reasoning in 
\cite[p.~402]{MR1460627}
relies on an identity of the kind
\begin{equation*}
- <c, v \cup u>=<c, u \cup v> 
\end{equation*}
but there is no reason for such an identity
to be valid since $u$ and $v$ are
(parabolic) 1-cocycles on $\pi$,
and parabolicity does not entail such an identity.

To fix this problem, 
in the statement of
\cite[Key Lemma 8.4 p.~397]{MR1460627},
replace identity  (8.4.2) with
\begin{equation}
\omega_V([v], [u]) = \tfrac 12( <c, u \cup v -v \cup u> +
\sum (X_j \cdot z_j Y_j -Y_j \cdot z_j X_j)).
\end{equation}
The proof of \cite[Theorem 8.3 p.~397]{MR1460627}
 works fine with this identity.

\section*{Acknowledgements}
The late Peter Slodowy taught me
geometric invariant theory.
See in particular  \cite{MR1044582} and the literature there;
this book devotes considerable space to Luna's slice theorem.
Also Peter Slodowy tried to hire me as a colleague in his department.
I am indebted to P. Heinzner for some email discussion
regarding the proof of
Proposition \ref{openness}(1)
and to
P. Boalch for a number of comments on a draft of the paper;
in particular Appendix I is due to P. Boalch.
I gratefully acknowledge support by the CNRS and by the
Labex CEMPI (ANR-11-LABX-0007-01).


\def\cprime{$'$} \def\cprime{$'$} \def\cprime{$'$} \def\cprime{$'$}
  \def\cprime{$'$} \def\cprime{$'$} \def\cprime{$'$} \def\cprime{$'$}
  \def\dbar{\leavevmode\hbox to 0pt{\hskip.2ex \accent"16\hss}d}
  \def\cprime{$'$} \def\cprime{$'$} \def\cprime{$'$} \def\cprime{$'$}
  \def\cprime{$'$} \def\Dbar{\leavevmode\lower.6ex\hbox to 0pt{\hskip-.23ex
  \accent"16\hss}D} \def\cftil#1{\ifmmode\setbox7\hbox{$\accent"5E#1$}\else
  \setbox7\hbox{\accent"5E#1}\penalty 10000\relax\fi\raise 1\ht7
  \hbox{\lower1.15ex\hbox to 1\wd7{\hss\accent"7E\hss}}\penalty 10000
  \hskip-1\wd7\penalty 10000\box7}
  \def\cfudot#1{\ifmmode\setbox7\hbox{$\accent"5E#1$}\else
  \setbox7\hbox{\accent"5E#1}\penalty 10000\relax\fi\raise 1\ht7
  \hbox{\raise.1ex\hbox to 1\wd7{\hss.\hss}}\penalty 10000 \hskip-1\wd7\penalty
  10000\box7} \def\polhk#1{\setbox0=\hbox{#1}{\ooalign{\hidewidth
  \lower1.5ex\hbox{`}\hidewidth\crcr\unhbox0}}}
  \def\polhk#1{\setbox0=\hbox{#1}{\ooalign{\hidewidth
  \lower1.5ex\hbox{`}\hidewidth\crcr\unhbox0}}}
  \def\polhk#1{\setbox0=\hbox{#1}{\ooalign{\hidewidth
  \lower1.5ex\hbox{`}\hidewidth\crcr\unhbox0}}}

\end{document}